\documentclass[11pt]{article}

\usepackage[T1]{fontenc}
\usepackage[utf8]{inputenc}

\usepackage{amsmath,amsfonts,latexsym, xspace,amsthm,graphicx, amssymb}

\usepackage{enumerate}

\usepackage[letterpaper,left=1in,right=1in,top=1in,bottom=1in]{geometry}
\usepackage{url}
\usepackage{tikz}

\usepackage{mathtools}
\mathtoolsset{showonlyrefs}

\usepackage{xcolor}
\definecolor{brightmaroon}{rgb}{0.76, 0.13, 0.28}
\definecolor{maroon}{rgb}{0.5, 0, 0}
\definecolor{webgreen}{rgb}{0, 0.5, 0}

\usepackage{hyperref} 
\hypersetup{colorlinks, breaklinks, urlcolor=maroon, linkcolor=maroon, citecolor=webgreen} 

\usepackage{subcaption}

\newcommand{\caiColor}{brightmaroon}

\newcommand{\fiColor}{blue}

\newcommand{\E}[1]{\mathbb{E} \left[#1\right]}

\newcommand{\Va}[1]{{\mathrm{Var}}\left(#1\right)}
\newcommand{\va}{{\mathrm{Var}}}

\newcommand{\mZeroIII}{\cM_{0,2}^{3}}
\newcommand{\mZeroI}{\cM_{0,2}^{1}}
\newcommand{\mZeroD}{\cM_{0,2}^{d}}

\newcommand{\norm}[1]{\left\lVert#1\right\rVert}

\newcommand{\ball}[1]{\hat{#1}}

\def\bI{\ball{I}}
\def\bX{\ball{X}}
\def\bY{\ball{Y}}
\def\bW{\ball{W}}

\newcommand{\bVec}{(\bX, \bY, \bW)}
\newcommand{\bVecI}{\left(\bX_n^{(i)}, \bY_{n}^{(i)}, \bW_{n}^{(i)}\right)}
\newcommand{\bVecIx}{\left(\bX^{(i)}, \bY^{(i)}, \bW^{(i)}\right)}
\newcommand{\bVecN}{(\bX_n, \bY_n, \bW_n)}

\newcommand{\nVec}{(X, Y, W)}
\newcommand{\nVecI}{\left(X_n^{(i)}, Y_{n}^{(i)}, W_{n}^{(i)}\right)}
\newcommand{\nVecN}{(X_n, Y_n, W_n)}

\def\bU{\ball{\Upsilon}}
\def\bZ{\ball{Z}}

\newcommand{\ind}{{\bf 1}}

\newcommand\cM{{\cal M}}
\newcommand\cN{{\cal N}}

\newcommand\cU{{\cal U}}
\newcommand\cV{{\cal V}}

\def\R{\mathbb{R}}
\def\N{\mathbb{N}}
\def\Z{\mathbb{Z}}
\def\QQ{Q}
\def\XX{X}
\def\YY{Y}
\def\ZZ{W}

\def\U{\Upsilon}
\def\F{\Phi}
\def\a{\alpha}
\def\b{\beta}

\def\G{\Gamma}

\def\z{\zeta}
\def\th{\theta}

\def\l{\lambda}
\def\m{\mu}

\def\r{\rho}
\def\s{\sigma}

\newcommand\Prob[1]{{\mathbb{P}\left\{#1\right\}}}

\newcommand{\dep}{d}

\newtheorem{theorem}{Theorem}
\numberwithin{theorem}{section}
\newtheorem{lemma}[theorem]{Lemma}
\newtheorem{corollary}[theorem]{Corollary}

\newcommand{\bfrac}[2]{\left({\frac{#1}{#2}}\right)}

\newtheorem{remark}[theorem]{Remark}

\newtheorem*{remark*}{Remark}

\theoremstyle{definition}
\newtheorem{example}[theorem]{Example}

\newcommand{\floor}[1]{\lfloor #1 \rfloor}
\newcommand{\ceil}[1]{\lceil #1 \rceil}

\newcommand{\dM}{d_2}

\newcommand{\dMto}{\overset{\dM}{\too}}
\newcommand{\inlaw}{\dto}
\newcommand{\inas}{\asto} 
\newcommand{\inLII}{\overset{L^{2}}{\too}}
\newcommand{\eqd}{\,{\buildrel {\mathrm{def}} \over =}\,}

\newcommand{\eql}{\eqdd}

\newcommand{\Unif}{\mathop{\mathrm{Unif}}}

\usepackage[square,numbers]{natbib}
\newcommand{\nBar}{\overline{n}}

\newcommand{\Toll}{D}
\newcommand{\bToll}{\ball{\Toll}}
\newcommand{\bTollN}{\bToll_{n}}
\newcommand{\TollN}{\Toll_{n}}

\numberwithin{equation}{section}

\newcommand{\refT}[1]{Theorem~\ref{#1}}

\newcommand{\refL}[1]{Lemma~\ref{#1}}
\newcommand{\refR}[1]{Remark~\ref{#1}}
\newcommand{\refS}[1]{Section~\ref{#1}}
\newcommand\kk{\varkappa}

\newcommand\set[1]{\ensuremath{\{#1\}}}

\newcommand\xpar[1]{(#1)}
\newcommand\bigpar[1]{\bigl(#1\bigr)}
\newcommand\Bigpar[1]{\Bigl(#1\Bigr)}

\newcommand\lrpar[1]{\left(#1\right)}

\newcommand{\too}{\longrightarrow}
\newcommand\dto{\overset{\mathrm{d}}{\too}}

\newcommand\asto{\overset{\mathrm{a.s.}}{\too}}
\newcommand\eqdd{\overset{\mathrm{d}}{=}}

\title{Inversions in split trees and conditional Galton--Watson trees}

\author{Xing Shi Cai, Cecilia Holmgren, Svante Janson, Tony Johansson, Fiona Skerman\thanks{This work was
        partially supported by two grants from the Knut and Alice Wallenberg Foundation, a grant
        from the Swedish Research Council, and the Swedish Foundations' starting grant from the
        Ragnar S\"{o}derberg Foundation. Fiona Skerman: work partially conducted while affiliated with the Heilbronn Institute for mathematical research at the University of Bristol. }\\
    Department of Mathematics, Uppsala University, Sweden\\
    \texttt{\small\{xingshi.cai, cecilia.holmgren, svante.janson, tony.johansson, fiona.skerman\}{@}math.uu.se}
}

\begin{document}

\maketitle

\begin{abstract}
    We study \(I(T)\), the number of inversions in a tree \(T\) with its vertices labeled uniformly
    at random, which is a generalization of inversions in permutations.
    We first show that the cumulants of \(I(T)\) have explicit formulas involving 
    the \(k\)-total common ancestors of \(T\) (an extension of the total path length).
    Then we consider \(X_n\), the normalized version of \(I(T_n)\), for a sequence of trees \(T_n\).
    For fixed \(T_{n}\)'s, we prove a sufficient condition for \(X_n\) to converge in distribution.
    As an application, we identify the limit of \(X_n\) for complete \(b\)-ary trees.
    For \(T_n\) being split trees \cite{MR1634354}, we show that \(X_n\)
    converges to the unique solution of a distributional equation.
    Finally, when \(T_n\)'s are conditional Galton--Watson trees, we show that \(X_n\) converges to
    a random variable defined in terms of Brownian excursions. 
    By exploiting the connection between inversions and the total path length, we are able to
    give results that significantly strengthen and broaden previous work by
    \citet{ps12}.

    \medskip\noindent\textbf{MSC classes:} 60C05
\end{abstract}

\section{Introduction}

\subsection{Inversions in a fixed tree}

Let \(\sigma_1, \dots, \sigma_n\) be a permutation of \(\{1,\dots,n\}\).  If \(i < j\) and
\(\sigma_i > \sigma_j\), then the pair \( (\sigma_i, \sigma_j)\) is called an inversion.  The
concept of inversions was introduced by \citet{cramer1750} (\citeyear{cramer1750}) due to its
connection with solving linear equations.  More recently, the study of inversions has been motivated
by its applications in the analysis of sorting algorithms, see, e.g.,
\cite[Section\ 5.1]{k98}. Many authors, including \citet[pp.\ 256]{f68}, \citet[pp.\ 29]{s97},
\citet{b73}, have shown that the number of inversions in uniform random permutations has a central
limit theorem.  More recently, \citet{m01} and \citet{lp03} studied permutations containing a fixed
number of inversions.

The concept of inversions can be generalized as follows.
Consider an unlabeled rooted tree $T$ on node set $V$. Let $\r$ denote the root. Write $u < v$
if $u$ is a \emph{proper ancestor} of $v$, i.e., the unique path from $\r$ to $v$ passes through $u$
and \(u \ne v\). Write
$u \le v$ if \(u\) is an ancestor of \(v\), i.e., either $u < v$ or $u = v$. Given a bijection $\l : V \to \{1,\dots,|V|\}$ (a {\em
    node labeling}), define the number of {\em inversions}
$$
I(T, \l) \eqd \sum_{u < v} {\ind}_{\l(u) > \l(v)}.
$$
Note that if \(T\) is a path, then \(I(T, \l)\) is nothing but the number of inversions in a
permutation. Our main object of study is the random variable $I(T)$, defined
by $I(T) = I(T, \l)$ where $\l$ is chosen uniformly at 
random from the set of bijections from $V$ to $\{1,\dots,|V|\}$.

The enumeration of trees with a fixed number of inversions has been studied by \citet{mr68} and
\citet{gsy95} using the so called \emph{inversions polynomial}.  While analyzing linear probing
hashing, \citet{fpv98} noticed that the numbers of inversions in Cayley
trees with uniform
random labeling converges to an Airy distribution.  \citet{ps12} showed that this is true for
conditional Galton--Watson trees, which encompasses the case of Cayley trees.

For a node \(v\), let $z_v$ denote the size of the subtree rooted at $v$. 
The following representation of \(I(T)\), proved in \refS{sec:prelude}, is the basis of most of our results.
\begin{lemma}\label{lem:independence}
    Let \(T\) be a fixed tree. Then
    \begin{equation}\label{it}
        I(T) \eqdd \sum_{v\in V} Z_v,
    \end{equation}
    where $\{Z_v\}_{v\in V}$ are independent random variables, and 
    $Z_v \sim \Unif\{0,1,\dots,z_v-1\}$.
\end{lemma}

We will generally be concerned with the centralized number of inversions, i.e., $I(T) - \E{I(T)}$.
For any $u < v$ we have $\Prob{\l(u) > \l(v)} = 1/2$.
Let \(h(v)\) denote the \emph{depth} of \(v\), i.e., the distance from \(v\) to the root \(\rho\).
(The \emph{distance} from \(u\) to \(v\) is the number of edges in the unique path connecting
\(u\) and \(v\).)
It immediately follows that,
\begin{equation}\label{eq:halftpl}
\E{I(T)} = \sum_{u < v} \E{{\ind}_{\l(u) > \l(v)}} = \frac12 \U(T)
,
\end{equation}
where $\U(T) \eqd \sum_{v} h(v)$ is called the {\em total path length} (or {\em internal path
    length}) of $T$.

Let $\kk_k=\kk_k(X)$ denote the $k$-th cumulant of a random variable $X$ (provided it exists); thus
$\kk_1(X)=\E X$ and $\kk_2(X)=\Va X$ (see \cite[Theorem\ 4.6.4]{g13}).  We now define \(\U_{k}(T)\),
\emph{the $k$-total common ancestors} of \(T\), which allows us to generalize \eqref{eq:halftpl} to higher
cumulants of \(I(T)\).
For \(k\) nodes $v_1, \ldots, v_k$ (not necessarily distinct), let $c(v_1, \ldots,v_k)$ be the
number of ancestors that they share, i.e.,
\[
    c(v_1,\dots,v_k) \eqd \left|
        \left\{
            u \in V: u \le v_{1}, u \le v_{2},\dots, u \le v_{k}
        \right\}
    \right|
    .
\]
We define
\begin{equation}\label{Uk}
   \Upsilon_{k}(T) \eqd \sum_{v_1,\dots,v_k} c(v_{1},\dots,v_{k}),
\end{equation}
where the sum is over all ordered $k$-tuples of nodes in the tree.  For a single node $v$, 
$h(v)=c(v)-1$, since \(v\) itself is counted in \(c(v)\). So $\U(T)=\U_1(T)-|V|$; i.e., we recover the
usual notion of the total path length.  
\begin{theorem}
    \label{thm:cumulant}
    Let $T$ be a fixed tree. 
    Let \(\kk_{k}(I(T))\) be the \(k\)-th cumulant of \(I(T)\).
    Then 
    \begin{align}
        \E{I(T)}& = \kk_{1}(I(T)) =\frac12\U(T)
        = \frac12(\U_{1}(T)-|V|), 
        \label{EIT}
        \\
        \Va{I(T)}& = \kk_{2}(I(T)) = \frac{1}{12}
        (\U_{2}(T)-|V|)
        \label{VIT},
    \end{align}
    and, more generally, for \(k \ge 1\),
    \begin{align}
        \kk_{2k+1}(I(T)) = 0,
        \qquad
        \kk_{2k}(I(T))=\frac{B_{2k}}{2k}
        (\U_{2k}(T)-|V|),
        \label{kkIT}
    \end{align}
    where \(B_k\) denotes the \(k\)-th Bernoulli number.
    Moreover, \(I(T)\) has the moment generating function
    \begin{equation}\label{mgfIT}
        \E{e^{tI(T)}}=\prod_{v\in V}\frac{e^{z_vt}-1}{z_v(e^t-1)},
    \end{equation}
and for the centralized variable  we have the estimate
    \begin{equation}\label{mgfIT*}
        \E{e^{t(I(T)-\E {I(T)})}}
        \le \exp\Bigpar{\tfrac{1}{8}t^2\sum_{v\in T}(z_v-1)^2}
        \le \exp\Bigpar{\tfrac{1}{8}t^2\sum_{v\in T}z_v^2}
        =\exp\Bigpar{\tfrac{1}{8}t^2\Upsilon_2(T)},
        \qquad t\in\R.
    \end{equation}
\end{theorem}

\begin{remark}\label{R1}
    Recalling that \(B_{1}=-1/2\) and \(B_{2k+1}=0\) for \(k \ge 1\), 
\eqref{EIT}--\eqref{kkIT} 
can also be written
    as
    \[
        \kk_{k}(I(T)) = \frac{B_k}{k}(-1)^{k}(\Upsilon_{k}(T)-|V|),
        \qquad
        k \ge 1
        .
    \]
\end{remark}

\begin{remark}
Higher moments and central moments 
can be calculated from the cumulants by standard formulas \cite{s95}.
(Note that all odd central moments vanish by symmetry.)
For example,  recalling $B_4=-1/30$, \refT{thm:cumulant} implies that
\begin{equation}
\E {\xpar{I(T)-\E{I(T)}}^4} 
= 
3\kk_2(I(T))^2 
+
\kk_4(I(T))
= 
\frac{1}{48} \Bigpar{
    \U_{2}(T)-|V|
}^2
-\frac{1}{120}
(
\U_{4}(T)-|V|
).
\end{equation}
\end{remark}

\begin{remark}
    An inversion is a special case of a pattern in a permutation. Thus, just as we can study
    inversions in trees, we can also study other patterns in trees. A recent paper by
    \citet{finoa2018} gives generalizes \ref{thm:cumulant} from inversions to any fixed patterns.
\end{remark}

\subsection{Inversions in sequences of trees}

The total path length $\U(T)$ has been studied for
random trees like split trees \cite{MR3025680} and conditional Galton--Watson trees \cite[Corollary
9]{a91}.  This leads us to focus on the deviation
$$
X_n 
=
\frac{I(T_n) - \E{I(T_n)}}{s(n)}
,
$$
under some appropriate scaling $s(n)$, for a sequence of (random or fixed)
trees $T_n$, where $T_n$ has size \(n\).

\subsubsection*{Fixed trees}

\begin{theorem}
    \label{thm:fixed}
    Let \(T_n\) be a sequence of fixed trees on $n$ nodes.
    Let
$$
    X_n =
    \frac{I(T_n)-\E {I(T_n)}}{\sqrt{\Upsilon_2(T_n)}}.
    $$
    Assume that for all \(k \ge 1\),
    \begin{equation}
        \frac{\Upsilon_{2k}(T_n)}{\Upsilon_2(T_n)^{k}}
        \to
        \zeta_{2k},
        \label{condition}
    \end{equation}
for some sequence \((\zeta_{2k})\).
    Then there exists a unique distribution $X$ with 
    \begin{equation}
        \kk_{2k-1}(X) = 0,
        \qquad
        \kk_{2k}(X) = \frac{B_{2k}}{2k} \zeta_{2k}, 
        \qquad k \ge 1,
    \end{equation}
    such that 
$X_n\dto X$ and, moreover,
\(\E{e^{tX_n}} \to \E{e^{tX}}<\infty\) for every $t \in \R$.
\end{theorem}

\begin{remark}
By \refT{thm:cumulant},
\( \Va{X_n} =  (\Upsilon_{2}(T_n)-n)/({12 s(n)^2}).\)
Thus, it is natural to consider
\(s(n) 
    = \Theta\bigpar{\sqrt{\Upsilon_{2}(T_n)-n}}
    = \Theta\bigpar{\sqrt{\Upsilon_{2}(T_n)}}
\), where we use \(\Upsilon_{2}(T_n) \eqd \sum_{v_1, v_2} c(v_1, v_2) \ge n^2\).
\end{remark}

\begin{remark}
    The functions \(\psi_{X_n}(t) \eqd \E{e^{tX_n}}\) and
    \(\psi_{X}(t) \eqd \E{e^{tX}}\) are called moment generating functions of \(X_n\) and \(X\) respectively.
The convergence \(\psi_{X_n}(t) \to \psi_{X}(t)<\infty\) in a neighborhood
    of \(0\) implies that \(X_{n} \inlaw X\) and  
    \( (|X_n|^{r})_{n \ge 1}\) is uniformly integrable for all \(r > 0\); 
thus \(\E{|X_{n}^r|} \to  \E{|X|^r}\) for all \(r > 0\)
and \(\E{X_{n}^r} \to  \E{X^r}\) for all integers \(r \ge1\).
See, e.g., \cite[Theorem 5.9.5]{g13}.
\end{remark}

As simple examples, we  consider two extreme cases.

\begin{example}
    When \(P_n\) is a path of \(n\) nodes, we have for fixed $k\geq 1$
    \[
        \Upsilon_{k}(P_n) \sim \frac{1}{k+1}n^{k+1}.
    \]
    Thus \(\Upsilon_{2k}(P_n)/\Upsilon_{2}(P_{n})^{k} \to \kk_{2k} =
    0\) for \(k \ge 2\).
    So by \refT{thm:fixed}, 
\(X_n\) converges to a normal distribution, and we recover the central 
    limit law for inversions in permutations.
    Also, the vertices have subtree sizes $1,\ldots, n$ and so we also recover 
from Theorem \ref{thm:cumulant}
    the moment generating
    function \(\prod_{j=1}^n ( e^{jt}-1)/(j(e^t-1))\) \cite{m01,s97}.
\end{example}

\begin{example}
  Let
$T_n=S_{n-1}$, a star with $n-1$ leaves, and denote the root by $o$.
We have $z_o=n$ and $z_v=1$ for $v\neq o$. Hence, by
\refL{lem:independence}, 
or directly,
$I(S_{n-1})\sim \Unif\set{0,\dots,n-1}$, and consequently
\begin{equation}
  \bigpar{I(T_n)-\E{I(T_n)}}/n \dto \Unif[-\tfrac12,\tfrac12].
\end{equation}
This follows also by \refT{thm:fixed}, since $\Upsilon_k(S_{n-1})\sim n^k$
for $k\ge2$ (e.g., by \refL{lem:common} below).
\end{example}

It is straightforward to compute the \(k\)-total common ancestors for \(b\)-ary trees. Thus our next result follows
immediately from \refT{thm:fixed}.
\begin{theorem}\label{thm:complete}
    Let $b\geq 2$ and let \(T_n\) be the complete $b$-ary tree 
    of height \(m\) with $n = (b^{m + 1} - 1)/(b-1)$ nodes.
    Let
    \[
        X_n 
        = \frac{I(T_n) - \E{I(T_n)}}{n}
        ,
        \qquad
        \text{and}
        \qquad
        X = \sum_{\dep \ge 0} \sum_{j=1}^{b^\dep} \frac{U_{\dep,j}}{b^\dep},
    \]
    where \( (U_{\dep, j})_{\dep \ge 0, j \ge 1}\) are independent \(\Unif[-1/2,1/2]\).
    Then $X_n\dto X$ and
    \(
        \E{e^{tX_n}} \to
        \E{e^{t X}} <\infty
    \),
    for every \(t \in \R\).
    Moreover \(X\) is the unique random variable with
    \begin{equation}\label{kkcomplete}
        \kk_{2k-1}(X) = 0, \qquad 
\kk_{2k}(X) = \frac{B_{2k}}{2k}\frac{b^{2k-1}}{b^{2k-1}-1}, \qquad k \ge 1.
    \end{equation}
\end{theorem}

\subsubsection*{Random trees}

We move on to random trees. We consider generating a random tree $T_n$ and, conditioning on $T_n$, labeling its nodes uniformly at random. The relation \eqref{eq:halftpl} is maintained for random trees:
\begin{equation}\label{eq:randomhalftpl}
\E{I(T_n)} = \E{\E{I(T_n)\mid T_n}} = \frac12\E{\U(T_n)}.
\end{equation}
The deviation of $I(T_n)$ from its mean can be taken to mean two different
things. Consider for some scaling function $s(n)$, 
\begin{equation}\label{eq:xy}
\XX_n = \frac{I(T_n) - \E{I(T_n)}}{s(n)}, \qquad 
\YY_n =  \frac{I(T_n) - \E{I(T_n)\mid T_n}}{s(n)}
=\frac{I(T_n) - \frac12\U(T_n)}{s(n)}.
\end{equation}
Then $\XX_n$ and $\YY_n$
each measure the deviation of $I(T_n)$, unconditionally and conditionally.
They are related by the identity 
\begin{equation}
  \label{eq:xy:z}
\XX_n = \YY_n + \ZZ_n/2,
\end{equation}
where
\begin{equation}\label{eq:zz}
\ZZ_n = \frac{\U(T_n) - \E{\U(T_n)}}{s(n)}.
\end{equation}

In the case of fixed trees \(\ZZ_{n} = 0\) and $\XX_n = \YY_n$, but for
random trees we 
consider the sequences separately.

We consider two classes of random trees --- split trees and conditional
Galton--Watson trees.

\subsubsection*{Split trees}

The first class of random trees which we study are split trees. They were introduced
by \citet{MR1634354} to encompass many families of trees that are frequently used in
algorithm analysis, e.g., binary search trees \cite{MR0142216}, $m$-ary search trees
\cite{MR0216622}, quad trees \cite{Finkel1974}, median-of-\((2k+1)\) trees \cite{w76}, 
fringe-balanced trees \cite{MR1236537}, digital search trees \cite{Coffman:1970} and random simplex
trees \cite[Example 5]{MR1634354}.

A split tree can be constructed as follows. Consider a rooted infinite
$b$-ary tree 
where
each node is a bucket of finite capacity $s$. 
We place $n$ balls at the root, and the balls
individually trickle down the tree in a random fashion until no bucket is above capacity. Each node draws a {\em split vector} $\cV = (V_1,\dots,V_b)$ from a common distribution, where $V_i$ describes the probability that a ball passing through the node continues to the $i$th child.  The trickle-down procedure is defined precisely
in Section \ref{sec:split}. Any node $u$ such that the subtree rooted as $u$ contains no balls
is then removed, and we consider the resulting tree $T_n$.

In the context of split trees we differentiate between $I(T_n)$ (the number of inversions on {\em
    nodes}), and $\bI(T_n)$ (the number of inversions on {\em balls}). In the former case, the nodes
(buckets) are given labels, while in the latter the individual balls are given labels. For balls
$\b_1,\b_2$, write $\b_1 < \b_2$ if the node containing $\b_1$ is a proper ancestor of the node
containing $\b_2$; if $\b_1,\b_2$ are contained in the same node we do not compare their labels.
Define
$$
\bI(T_n) = \sum_{\b_1 < \b_2} {\ind}_{\l(\b_1) > \l(\b_2)}.
$$
Similarly define $\bU(T_n)$ as the total path length on balls,
i.e., the sum of the depth of all balls.
And let
\begin{equation}\label{eq:ballxyz}
\bX_n = \frac{\bI(T_n) - \E{\bI(T_n)}}{n}, \quad \bY_n = \frac{\bI(T_n) - s_0\bU(T_n)/2}{n}, \quad \bW_n = \frac{\bU(T_n) - \E{\bU(T_n)}}{n}.
\end{equation}
Here $s_0$ is a fixed integer denoting the number of balls in any internal node, and we have $\bX_n
= \bY_n + s_0\bW_n/2$ (formally justified in Section \ref{sec:split}). The following theorem gives
the limiting
distributions of the random vector $(\bX_n, \bY_n, \bW_n)$. In Section \ref{sec:splitnodes} we state a similar result for
$(\XX_n, \YY_n, \ZZ_n)$ under stronger assumptions. Note that the concepts are identical for any class
of split trees where each node holds exactly one ball, such as binary search trees, quad trees,
digital search trees and random simplex trees.

Let \(\dM\) denote the Mallows metric, also called the minimal $\ell_2$ metric (defined in Section
\ref{sec:split}).  Let \(\mZeroD\) be the set of probability measures on
\(\R^{d}\) with zero mean and finite 
second moment.

\begin{theorem}\label{thm:split}
Let \(T_n\) be a split tree and let $\cV = (V_1,\dots,V_b)$ be a split vector. Define
$$
\m = -\sum_{i=1}^b \E{V_i\ln V_i}, \qquad \text{and} \qquad \Toll(\cV) = \frac{1}{\m} \sum_{i=1}^b V_i\ln V_i.
$$
Assume that \(\Prob{\exists i:V_i = 1} < 1\) and \(s_{0} > 0\). 
Let $\bVec$ be the unique solution in \(\mZeroIII\) for the
system of
fixed-point equations
\begin{equation}
    \left[
    \begin{aligned}
        & \bX \\
        & \bY \\
        & \bW
    \end{aligned}
    \right]
    \eql
    \left[
    \begin{aligned}
        & \sum_{i=1}^b V_i \bX^{(i)} + \sum_{j=1}^{s_0} U_j + \frac{s_0}{2} \Toll(\cV)  \\
        & \sum_{i=1}^b V_i \bY^{(i)} + \sum_{j=1}^{s_0} (U_j-1/2) \\
        & \sum_{i=1}^b V_i \bW^{(i)} + 1 + \Toll(\cV)
    \end{aligned}
    \right]
    .
    \label{eq:Y}
\end{equation}
Here \( (V_{1},\dots,V_{b})\), \(U_{1},\dots,U_{s_{0}}\), 
\(
    (\bX^{(1)}, \bY^{(1)}, \bW^{(1)}),
    \dots,
    (\bX^{(b)}, \bY^{(b)}, \bW^{(b)})
\)
are independent,
with \(U_{j} \sim \Unif[0,1]\) for \(j=1,\dots,s_0\), and
\(\bVecI \sim \bVec\) for \(i = 1,\dots,b\).
Then the sequence $\bVecN$ defined in
\eqref{eq:ballxyz} converges to \(\bVec\) in \(\dM\) and in moment generating function within a
 neighborhood of the origin.
\end{theorem}

The proof of Theorem \ref{thm:split} uses the contraction method, introduced by \citet{MR1104413}
for finding the total path length of binary search trees. 
The technique has been applied to 
$d$-dimensional quad trees by \citet{nr99} and to split trees in general by \citet{MR3025680}.  The
contraction method also has many other applications in the analysis of recursive algorithms, see,
e.g., \cite{MR2023025, MR1887306, MR1887296}. 

\begin{remark}
    We assume that \(s_0 > 0\), for otherwise we trivially have \(\bX_{n} = 0\) and \refT{thm:split}
    reduces to Theorem 2.1 in \cite{MR3025680}.
    \label{rk:s:knot}
\end{remark}

\begin{remark}
    In a recent paper, \citet{s17} showed that preferential attachment trees and random recursive
    trees can be viewed as split trees with infinite-dimensional split vectors. Thus we conjecture
    that the contraction method should also be applicable for these models and give results similar
    to \refT{thm:split}.
\end{remark}

\begin{remark}
Assume that the 
constant 
split vector $\cV = (1/b,\dots,1/b)$ is used and each node holds exactly one
ball (a
special case of digital search trees, see \cite[Example 7]{MR1236537}).
Then $\Toll(\cV) = -1$ and
\eqref{eq:Y} has the unique solution $\bVec = (X, X, 0)$, where \(X\) has the limiting
distribution for inversions in complete $b$-ary trees (see \refT{thm:complete}). This is as expected, as the
shape of a split tree with these parameters is likely to be very similar to a complete $b$-ary tree.
\end{remark}

\subsubsection*{Conditional Galton--Watson trees}

Finally, we consider conditional Galton--Watson trees (or equivalently, simply generated
trees), which were introduced by \citet{bienayme} and \citet{watson1875} to
model the evolution of populations. A Galton--Watson tree starts with a root node. Then recursively,
each node in the tree is given a random number of child nodes. 
The numbers of children are drawn independently from the same distribution \(\xi\) called the \emph{offspring
distribution}.

A conditional Galton--Watson tree \(T_n\) is a Galton--Watson tree conditioned on having \(n\) nodes.
It generalizes many uniform random tree models, e.g., Cayley trees, Catalan trees, binary trees,
\(b\)-ary trees, and Motzkin trees. For a comprehensive survey, see \citet{j12}.
For recent developments, see \cite{cai16, dhs17, j16, k17}.

In a series of three seminal papers, \citeauthor{a91} showed that \(T_n\) converges under re-scaling to
a \emph{continuum random tree}, which is a tree-like object constructed from a Brownian excursion
\cite{MR1085326, a91, MR1207226}.  Therefore, many asymptotic properties of conditional
Galton--Watson trees, such as the height and the total path length, can be derived from properties of
Brownian excursions \cite{a91}. Our analysis of inversions follows a similar route.  In
particular, we relate \(I(T_n)\) to the \emph{Brownian snake} studied by 
e.g., \citet{jm05}.

In the context of Galton--Watson trees, Aldous \cite[Corollary 9]{a91} showed that $n^{-3/2}\U(T_n)$
converges to an Airy distribution. We will see that the standard deviation of $I(T_n) -
\frac12\U(T_n)$ is of order
$n^{5/4} \ll n^{3/2}$, which by the decomposition \eqref{eq:xy:z} implies that
$n^{-3/2}I(T_n)$ converges to the same Airy distribution, recovering one of the main results of
\citet[Theorem 5.3]{ps12}. Our contribution for conditional Galton--Watson trees
is a detailed analysis of $\YY_n$ under the scaling function $s(n) = n^{5/4}$.

Let $e(s), s\in[0,1]$ be the random path of a standard Brownian excursion,
and define $C(s, t)\eqd \allowbreak C(t,s) \eqd 2\min_{s\leq u\leq t} e(u)$ for
$0\le s\le t\le 1$.

We define a random variable, see \cite{j02},
\begin{equation}\label{eq:etadef} 
    \eta \eqd \int_{[0,1]^2} C(s, t) \mathrm ds\ \mathrm dt
=4\int_{0\le s\le t\le 1}\min_{s\leq u\leq t} e(u).
\end{equation} 

\begin{theorem}\label{thm:galtonwatson}
Suppose $T_n$ is a conditional Galton--Watson tree with offspring distribution $\xi$ such that
$\E{\xi} = 1$, $\Va{\xi} = \s^2 \in (0, \infty)$, 
and $\E{e^{\alpha\xi}}<\infty$ for some $\a>0$, 
and define
$$
\YY_n = \frac{I(T_n) - \frac12 \U(T_n)}{n^{5/4}}.
$$
Then we have
\begin{equation}\label{gw}
  \YY_n \inlaw \YY \eqd \frac{1}{\sqrt{12\s}} \sqrt{\eta}\ \cN,
\end{equation}
where $\cN$ is a standard normal random variable, independent from the random variable $\eta$
defined in \eqref{eq:etadef}.
Moreover, \(\E{e^{t\YY_n}} \to \E{e^{t\YY}} < \infty\) for all fixed \(t \in \R\). 
\end{theorem}

The moments of $\eta$ and $Y$ are known \cite{MR2108865}, see
Section \ref{sec:galtonwatson}.

The rest of the paper is organized as follows.  In Section \ref{sec:prelude}, we prove
\refL{lem:independence} and \refT{thm:cumulant}. The results for fixed trees (Theorems
\ref{thm:fixed}, \ref{thm:complete}) are presented in Section \ref{sec:gen}.
Split trees and conditional Galton--Watson trees are considered in Sections \ref{sec:split} and
\ref{sec:galtonwatson} respectively. Sections \ref{sec:split} and \ref{sec:galtonwatson} are
essentially self-contained, and the interested reader may skip ahead.

\section{A fixed tree}\label{sec:prelude}

In this section we study a fixed, non-random tree $T$.  We begin with proving 
\refL{lem:independence}, which shows that
\(I(T)\) is a sum of independent uniform random variables.

\begin{proof}[Proof of {\refL{lem:independence}}]
We define $Z_u = \sum_{v :  v> u} {\ind}_{\l(u) > \l(v)}$ and note that
\begin{equation}
I(T) \eqd \sum_{u < v} {\ind}_{\l(u) > \l(v)} = \sum_{u\in V} \left(\sum_{v:v > u} {\ind}_{\l(u) > \l(v)}\right) = \sum_{u\in V} Z_u,
\end{equation}
showing \eqref{it}. 
Let $T_u\subseteq T$ denote the subtree rooted at $u$. It is clear that 
conditioned on the set $\l(T_u)$,
$\l$ restricted to $T_u$ is a uniformly random labeling of $T_u$ into
$\l(T_u)$.
Recall that \(z_u\) denotes the size of \(T_u\).
If the elements of $\l(T_u)$ are $\ell_1 < \dots < \ell_{z_u}$
and if $\l(u) = \ell_i$, then $Z_u = i-1$. As $\l(u)$ is uniformly distributed,
so is $Z_u$. 

We prove independence of the $Z_v$ by induction on $V$. The base case $|V| = 1$ is trivial.
Let $T_1,\dots,T_d$ be the subtrees rooted at the children of the root
$\rho$, and 
condition on the sets $\l(T_1),\dots,\l(T_d)$.
Given these sets, $\l$ restricted to $T_i$ is a uniformly random labeling
of $T_i$ using the given labels $\l(T_i)$, and these labelings are
independent for different $i$.
Hence, conditioning on $\l(T_1),\dots,\l(T_d)$, the $d$ families
$(Z_v)_{v\in T_i}$ are independent, and each is distributed as the
corresponding family for the tree $T_i$.

Consequently, by induction, still conditioned on $\l(T_1),\dots,\l(T_d)$,
 $(Z_v)_{v\neq\rho}$ are independent, with 
$Z_v \sim \Unif\{0,1,\dots,z_v-1\}$. 
Furthermore, $Z_\rho=\l(\rho)-1$, and $\l(\rho)$ is determined by
$\l(T_1),\dots,\l(T_d)$ (as the only label not in $\bigcup_1^d\l(T_i)$).
Hence the family  $(Z_v)_{v\neq\rho}$ of  independent random variables is
also independent of $Z_\rho$, and thus 
 $(Z_v)_{v\in V}$ are independent.
This completes the induction, and thus the proof.
\end{proof}

Our first use of the representation in \refL{lem:independence} is to prove \refT{thm:cumulant},
which gives both a formula for the moment generating function and explicit formulas for the
cumulants of $I(T)$ for a fixed $T$.  The proof begins with a simple lemma giving the cumulants and
the moment generating function of $Z_v$ in \refL{lem:independence}, from which \refT{thm:cumulant} will
follow immediately.

Recall that the Bernoulli numbers $B_k$ can be defined by their generating
function
\begin{equation}\label{Bernoulli}
  \sum_{k=0}^\infty B_k\frac{x^k}{k!} = \frac{x}{e^x-1}
\end{equation}
(convergent for $|x|<2\pi$), see, e.g., \cite[(24.2.1)]{NIST:DLMF}.
Recall also $B_0=1$, $B_1=-\frac12$ and $B_2=\frac{1}6$, and that
$B_{2k+1}=0$ for $k\ge1$.

\begin{lemma}\label{lem:Z}
Let $N\ge1$, and  let $Z_N$ be uniformly distributed on
$\{0,1,\dots,N-1\}$. Then 
$\E {Z_N}= (N-1)/2$,
$\Va {Z_N}=(N^2-1)/12$
and, more generally,
\begin{equation}\label{kk}
\kk_k(Z_N)=\frac{B_k}{k}(N^k-1),
\qquad k\ge2,
\end{equation}
where $B_k$ is the $k$-th Bernoulli number.
The moment generating function of \(Z_N\) is
\begin{equation}\label{mgfZ}
  \E{e^{tZ_N}} 
= \frac{e^{Nt}-1}{N(e^t-1)}.
\end{equation}
\end{lemma}

\begin{proof}
This is presumably well-known, but we include a proof for completeness.
The moment generating function of $Z_N$ is
\begin{equation}\label{mgfZ2}
  \E{e^{tZ_N}} = \frac{1}N\sum_{j=0}^{N-1} e^{jt} 
=  \frac{e^{Nt}-1}{N(e^t-1)},
\end{equation}
verifying \eqref{mgfZ}.
The function $(e^t-1)/t$ is analytic and non-zero in the disc $|t|<2\pi$,
and thus has there a well-defined analytic logarithm
\begin{equation}\label{ft}
  f(t):=\log\frac{e^t-1}{t},
\end{equation}
with $f(0)=0$.
By \eqref{mgfZ2} and \eqref{ft}, 
the cumulant generating function of $Z_N$ can be written as
\begin{equation}\label{psi2}
  \log\E{e^{tZ_N}} = f(Nt)-f(t).
\end{equation}
Differentiating \eqref{ft} yields (for $0<|t|<2\pi$)
\begin{equation}
  f'(t) = \frac{d}{dt}\bigpar{\log(e^t-1)-\log t}
=\frac{e^t}{e^t-1}-\frac{1}t
=\frac{1}{e^t-1}+1-\frac{1}t
,
\end{equation}
and thus, using \eqref{Bernoulli}, 
\begin{equation}
 t f'(t)
=\frac{t}{e^t-1}+t-1
=  \sum_{k=0}^\infty B_k\frac{t^k}{k!} -1+t 
=  \sum_{k=2}^\infty B_k\frac{t^k}{k!} +\frac12 t. 
\end{equation}
Consequently,
\begin{equation}\label{ftx}
 f(t)
=  \sum_{k=2}^\infty \frac{B_k}{k}\frac{t^k}{k!} +\frac12 t,
\end{equation}
and thus by \eqref{psi2}
\begin{equation}
\log\E{e^{tZ_N}}
=  \sum_{k=2}^\infty \frac{B_k}{k}(N^k-1)\frac{t^k}{k!} +\frac{N-1}2 t 
. 
\end{equation}
The results on cumulants follow. 
(Of course, $\E{Z_N}$ is more  simply calculated directly.)
\end{proof}

\begin{remark}\label{RU}
  Similarly, using \eqref{ftx}, or by \eqref{kk} and a limiting argument,
if $U\sim\Unif[0,1]$ or $U\sim\Unif[-\frac12,\frac12]$, then
$\kk_k(U)=B_k/k$, $k\ge2$.
\end{remark}

Recall that in the introduction, we defined
\[
    c(v_{1},\dots,v_k) \eqd |\{u:u\le v_1, \dots, u \le v_k\}|,
\]
i.e., \(c(v_{1},\dots,v_{k})\) is the number of common ancestors of \(v_{1},\dots,v_{k}\).
\begin{lemma}
    \label{lem:common}
    Let \(z_v\) denote the number of vertices in subtree rooted at \(v\). Then for $k\geq 1$,
    \[
        \sum_{v} z_v^{k} = 
        \U_{k}(T) \eqd
        \sum_{v_1,\dots,v_k} c(v_{1},\dots,v_{k})
        .
    \]
\end{lemma}

\begin{proof}
It is easily seen that
\begin{equation}
    \sum_{u} z_u = \sum_{u, v} {\ind}_{\{u \le v\}} = \sum_{v} c(v) 
    .
\end{equation}
Similarly, 
\begin{equation}
    \sum_{u} z_u^2 = \sum_{u,v,w}{\ind}_{\{u \le v, u \le w\}} = \sum_{v,w} c(v, w).
\end{equation}
More generally, 
\begin{equation*}
    \sum_{u}z_u^k 
    = \sum_u \prod_{i=1}^k \left(\sum_{v_i} {\ind}_{\{u \le v_i\}}\right) 
    = \sum_{v_1,\dots,v_k} c(v_{1},\dots,v_{k})
    .
    \qedhere
\end{equation*}
\end{proof}

\begin{remark}
Observe that all common ancestors of the $k$
vertices must lie on a path; stretching from the last common ancestor to the root. Define a related
parameter $\U'_k(T)$ to be 
the sum over all $k$-tuples of
the length of this path  (rather than number of vertices in the path). 
We call this the \emph{$k$-common path length}. Now $\U'_1(T)=\U(T)$ and
$\U'_2(T)$ has appeared in various contexts, see for example \cite{j02}
(where it is denoted $Q(T)$).
Let $v_1 \land v_2$ denote the last common ancestor of the vertices $v_1$
and $v_2$. 
It is easy to see that, with $n=|T|$,
\[
    \Upsilon'_{k}(T) 
    \eqd
\sum_{v_1,\dots,v_k} h(v_{1} \land \dots \land v_{k})
    =
    \sum_{v_1,\dots,v_k} (c(v_{1},\dots,v_{k}) -1)   
    =
     \U_{k}(T) - n^{k},\]
and by Lemma~\ref{lem:common}, $\U_{k}(T)=\sum_v z_v^k$, so $\U'_{k}(T)=\sum_{v\neq \rho} z_v^k$.
\end{remark}

\begin{remark}
Let $S_k$ be a star with $k$ leaves $\ell_1, \dots,\ell_k$ and root
$o$. Then $\U_k(T)$ is the number of embeddings $\phi:V(S_k)\rightarrow V(T)$ such that
$\phi(o) \leq \phi(\ell_i)$ for each $i$.  Similarly the $k$-common path-length $\U'_k(T)$ is the
number of such embeddings $\phi$ such that $\phi(o) < \phi(\ell_i)$ for each $i$.
\end{remark}

\begin{proof}[Proof of {\refT{thm:cumulant}}]
    Since cumulants are additive 
for sums of independent random variables,
    an immediate consequence of Lemmas \ref{lem:independence} and
    \ref{lem:Z} is that
\begin{equation}
        \kk_k(I(T))=\frac{B_k}{k}
        \sum_{v\in V} \bigpar{z_v^k-1}
        = \frac{B_k}{k}(\U_k(T)-|V|)
        ,
        \qquad k\ge1.
        \label{eq:cum:z}
    \end{equation}
    where the last equality follows from \refL{lem:common}.
    The fact that $\E{I(T)}=\frac12\U(T)$ was noted already in
    \eqref{eq:halftpl}.

Similarly, \eqref{mgfIT} follows from
\refL{lem:independence} and \eqref{mgfZ2}.

For the estimate \eqref{mgfIT*}, note first, e.g.\ by Taylor expansions,
that $\cosh x \le e^{x^2/2}$ for every real $x$.
It follows that if $U$ is any symmetric random variable with $|U|\le a$,
then 
\begin{equation}\label{etu}
\E {e^{tU}}=\E {\cosh(tU)}\le e^{a^2t^2/2}.  
\end{equation}
(See \cite[(4.16)]{Hoeffding} for a more general result.)
\refL{lem:independence} thus implies,
applying \eqref{etu} to each $Z_v-\E{Z_v}$, 
\begin{equation}
  \E{e^{t(I(T)-\E{I(T)}}} 
=\prod_v\E{e^{t(Z_v-\E{Z_v})}}
\le \prod_v e^{t^2(z_v-1)^2/8}
=
e^{t^2\sum_v(z_v-1)^2/8},
\end{equation}
which yields \eqref{mgfIT*}, using also \refL{lem:common}.
\end{proof}

\section{A sequence of fixed trees}\label{sec:gen}

In this section, we study
$$
X_n 
=
\frac{I(T_n) - \E{I(T_n)}}{s(n)}
,
$$
where \(T_n\) is a sequence of fixed trees and \(s(n)\) is an appropriate normalization factor.  
We
start by proving \refT{thm:fixed}, a sufficient condition for \(X_n\) to
converge in distribution 
when \(s(n) = \sqrt{\Upsilon_{2}(T_n)}\).

\begin{proof}[Proof of {\refT{thm:fixed}}]
    First \(\kk_{1}(X_n) = \E{X_n} = 0\).
         For \(k \ge 2\), note that shifting a random variable does not change its \(k\)-th cumulant.
         Also note that \(\Upsilon_k(T_n) \eqd \sum_{v_1,\dots,v_k} c(v_1,\dots,v_k) \ge n^k\).
        Therefore, it follows from \refT{thm:cumulant} that
    \[
        \kk_{k}(X_n) 
        = \frac{ \kk_{k}(I(T_n)) }{(\Upsilon_{2}(T_n)-n)^{k/2}}
        = \frac{B_k}{k} \frac{\Upsilon_{k}(T_n)-n}{(\Upsilon_{2}(T_n)-n)^{k/2}}
        \sim \frac{B_k}{k} \frac{\Upsilon_{k}(T_n)}{\Upsilon_{2}(T_n)^{k/2}}
        , 
        \qquad
        k \ge 2.
    \]
    Recall that all odd Bernoulli numbers except \(B_1\) are zero. 
    Thus letting \(\zeta_k=0\) for all odd \(k\), 
    the assumption 
    that \( \Upsilon_{2k}(T_n)/\Upsilon_{2}(T_n)^{k} \to \zeta_{2k}\) for all $k \ge 1$
    implies that
    \[
        \kk_{k}(X_n) \to \frac{B_k}{k} \zeta_k, \qquad k \ge 1.
    \]
Since every moment can be expressed as a polynomial in cumulants, it follows
that every moment $\E{X_n^k}$ converges, $k\ge1$.
Thus to show that there exists an \(X\) such that \(X_n \inlaw X\), it suffices to show that the
moment generating function $\E{ e^{t X_n}}$ stays bounded for all small fixed
$t$;
we shall show that this holds for all real $t$.
In fact, using \refL{lem:common},
\begin{equation}
  \sum_v(z_v-1)^2
\le   \sum_v(z_v^2-1)
=
\U_2(T_n)-n
\le
\U_2(T_n)
.
\end{equation}
Hence, \eqref{mgfIT*} yields
\begin{equation}
  \E{e^{tX_n}}
\le \exp\Bigpar{\tfrac18 \bigpar{t/\sqrt{\U_2(T_n)}}^2\sum_v(z_v-1)^2}
\le \exp\bigpar{\tfrac18 t^2},
\qquad t\in\R.
\end{equation}
This and the moment convergence imply the claims in the theorem.
\end{proof}

\subsection{The complete \texorpdfstring{$b$}{b}-ary tree}\label{sec:complete}

We prove Theorem \ref{thm:complete}, which asserts that for complete $b$-ary
trees the limiting variable of $X_n$ is the unique $X$ for which $\kk_k(X) =
\frac{B_k}{k}\frac{b^{k-1}}{b^{k-1}-1}$ for even $k\geq 2$ and zero for
odd
 $k$. Fix $b\geq 2$.  In the complete $b$-ary tree 
of height $m$,
each node $v$ at depth $d\in\{0,1,\dots,m\}$ has subtree size $z_v =
a_{m, d} \eqd (b^{m-d+1}-1)/(b-1)$. Hence Lemma \ref{lem:independence} implies that
\(
X_n 
= \sum_{\dep=0}^m \sum_{j=1}^{b^\dep} {Z_{\dep, j}}/{n},
\)
where
\[
Z_{\dep, j} \sim \Unif
\left\{-\frac{a_{m, \dep}-1}{2},
-\frac{a_{m, \dep}-2}{2},
\dots,
\frac{a_{m, \dep}-2}{2},
\frac{a_{m, \dep}-1}{2}
\right\}
\]
are independent random variables.
Let \(U_{\dep,j}\) be independent \(\Unif[-\frac{1}{2},
\frac{1}{2}]\). Approximating $Z_{d, j} \approx U_{d, j}a_{m, d}$ and noticing that $n/a_{m,\dep}\approx b^d$, intuitively we should have for large $n$,
\begin{equation}
    X_n =
    \sum_{\dep=0}^m \sum_{j=1}^{b^\dep} \frac{a_{m,\dep}}{n} 
\cdot\frac{Z_{\dep, j}}{a_{m,\dep}}
    \approx 
    \sum_{\dep \ge 0} \sum_{j=1}^{b^\dep} \frac{U_{\dep,j}}{b^\dep} 
    \eqd X.
    \label{eq:b:limit}
\end{equation}

It is not difficult 
to show this rigorously by truncating the sums.
Also,
it is not difficult 
to prove \refT{thm:complete} by showing that \(\E{e^{tX_n}} \to \E{e^{tX}}\) for
all \(t \in \R\) and checking the cumulants of \(X\),
using \refR{RU}.
But instead we choose the  route of computing
the \(k\)-total common ancestors of \(b\)-ary trees and then applying
\refT{thm:fixed}.

\begin{lemma}
    \label{lem:B:common}
    Assume \(b \ge 2\).
    Let \(T_n\) be the complete \(b\)-ary tree on \(n = (b^{m+1} - 1)/(b-1)\) nodes.
    Then
    \[
        \Upsilon_{1}(T_n) \sim n \log_{b}n, 
        \qquad
        \Upsilon_{k}(T_n) \sim \frac{b^{k-1}}{b^{k-1}-1} n^{k},
        \qquad
        k \ge 2
        .
    \]
\end{lemma}

\begin{proof}
The height of \(T_n\) is  \(m \sim \log_b n \).
    It follows from \refL{lem:common} that
    \begin{align*}
        \Upsilon_{1}(T_n) 
        = 
        \sum_{v} z_v
        =
        \sum_{d=0}^{m} b^{d} 
        \times
        a_{m,d}
        =
        \frac{b^{m+1}}{b-1}
        \sum_{d=0}^{m} \left( 1 - \frac{1}{b^{m+1-d}} \right)
        =
        \frac{b^{m+1}}{b-1}
        \left( m + O(1) \right)
        \sim n \log_{b}n
        .
    \end{align*}
    Similarly, for any fixed \(k \ge 2\),
    \[
        \Upsilon_{k}(T_n) 
        = 
        \sum_{v} z_v^{k}
        =
        \sum_{d=0}^{m} b^{d} 
        \times a_{m,d}^{k}
        =
        \frac{b^{(m+1)k}}{(b-1)^{k}}
        \sum_{d=0}^{m} \frac{1}{b^{d(k-1)}} \left(1 - \frac{1}{b^{m+1-d}} \right)^{k}
        \sim
        n^{k} \frac{b^{k-1}}{b^{k-1}-1}
        .
        \qedhere
    \]
\end{proof}

\begin{proof}[Proof of {\refT{thm:complete}}]
    Let \(X_n' = (I(T_n)-\E{I(T_n)})/\sqrt{\Upsilon_2(T_n)}\).
    By \refL{lem:B:common}, for fixed \(k \ge 1\),
    \[
        \frac{\Upsilon_{2k}(T_{n})}{\Upsilon_{2}(T_n)^k}
        \sim
        \frac{n^{2k}\frac{b^{2k-1}}{b^{2k-1}-1}}{\left(n^2\frac{b}{b-1}\right)^{k}}
=
        \frac{b^{2k-1}}{b^{2k-1}-1} \left( \frac{b-1}{b} \right)^{k}.
    \]
    By \refT{thm:fixed}, there exists a unique distribution \(X'\) such that
    \[
        \kk_{2k-1}(X') = 0, 
        \qquad
        \kk_{2k}(X') = 
        \frac{B_{2k}}{2k}
        \frac{b^{2k-1}}{b^{2k-1}-1} \left( \frac{b-1}{b} \right)^{k},
        \qquad
        k \ge 1;
    \]
moreover, \(\E{e^{tX_n'}} \to \E{e^{tX'}}<\infty\) for every $t$.
Recall that,
using \refL{lem:B:common} again,
\[
        X_n \eqd \frac{I(T_n)-\E{I(T_n)}}{n} 
= (1+o(1)) \Bigpar{\frac{b}{b-1}}^{1/2} X_n'.
    \]
    Let \(X'' = \bigpar{b/(b-1)}^{1/2}X'\);
then \(\E{e^{tX_n}} \to \E{e^{tX''}}\) for
    every real $t$ and
    \(X''\) has cumulants
\[
        \kk_{1}(X'') = 0, 
        \qquad
        \kk_{k}(X'') = 
        \frac{B_k}{k}
        \frac{b^{k-1}}{b^{k-1}-1},
        \qquad
        k \ge 2,
    \]
    as in \eqref{kkcomplete}. 
It is not difficult to show that \(X''\) has the same distribution as \(X\)
defined  in \eqref{eq:b:limit} by checking the cumulants of \(X\), using
\refR{RU}.
\end{proof}

\subsection{Balanced \texorpdfstring{\(b\)}{b}-ary trees}

We call a \(b\)-ary tree {\em balanced} if all but the last level of the tree is full and vertices at
the last level take the leftmost positions.  A simple example of a
balanced binary tree is \(T_n\)
in which both the left and right subtrees are complete \(b\)-ary trees but the left subtree has one more level
than the right subtree.  Since the left subtree is of size about \(2n/3\), and the right subtree is of
size about \(n/3\), Theorem \ref{thm:complete} 
and \refL{lem:independence}
imply that
\[
    X_n = \frac{I(T_n)-\E{I(T_n)}}{n}
    \inlaw U + \frac{2 X'}{3} + \frac{X''}{3},
\]
where \(U \sim \Unif[-\frac{1}{2},\frac{1}{2}]\) and \(X', X''\) are independent copies of
\(X\). The three terms in the limit correspond to inversions involving the root, inversions in
the left subtree and inversions in the right subtree.

The above example shows that the limit distribution of \(X_n\) in a balanced \(b\)-ary tree
in which each subtree of the root is complete should be 
\(U\) plus a linear combination of independent copies of \(X\).
We formalize this observation in the following corollary.
\begin{corollary}
    Let \(T_n\) be a balanced \(b\)-ary tree. 
    Let \(X_n\) and \(X\) be as in Theorem \ref{thm:complete}.
    Let \(\{x\} \eqd x-\floor{x}\).
    Assume that
\begin{equation}
        \{\log_{b}\left( (b-1)n \right)\}
        =
        \log_{b} \left( 1+\frac{b-1}{b}i \right)
        +
        o\left( \frac{1}{\log n} \right)
        ,
        \label{eq:almost:complete}
    \end{equation}
    where \(i \in \{0,\dots,b\}\) is a constant.
    We have
    \[
        X_{n} 
        \inlaw 
        U 
        + \sum_{j=1}^{i} \frac{b}{b+i(b-1)} X^{(j)}
        + \sum_{j=i+1}^{b} \frac{1}{b+i(b-1)} X^{(j)}
        \eqd X(b,i)
        ,
    \]
    where \(U \sim \Unif[-\frac{1}{2},\frac{1}{2}]\), \(X^{(j)}\sim X\) are all independent.
    Moreover \(\E{e^{tX_{n}}} \to \E{e^{tX(b,i)}}\) for all \(t \in \R\).
\end{corollary}

\begin{remark}
Condition \eqref{eq:almost:complete} is equivalent of saying that all the
\(b\) subtrees of the root of \(T_n\) except one (either the \(i\)-th or the \((i+1)\)-th) are
complete \(b\)-ary trees and the exceptional subtree differs from a complete \(b\)-ary tree in
size by at most \(o(n/\log(n))\).
\end{remark}

\section{A sequence of split trees}\label{sec:split}

We will now define split trees introduced by Devroye \cite{MR1634354}.
The random split tree $T_n$ has parameters $b, s, s_0, s_1, \cV$ and $n$. The integers $b, s, s_0, s_1$ are required to satisfy the inequalities
\begin{equation}
    2 \le b, \quad 0 < s, \quad 0\leq s_0\leq s, \quad 0\leq bs_1\leq s+1-s_0.
    \label{eq:split:para}
\end{equation}
and $\cV=(V_1,\dots,V_b)$ is a random non-negative vector with 
\(\sum_{i=1}^b V_i = 1\).
We define $T_n$ algorithmically. 
Consider the infinite $b$-ary tree $\cU$, and view each node as a
bucket with capacity $s$. Each node $u$ is assigned an independent copy 
$\cV_u$ of the random split vector $\cV$.
Let
$C(u)$ denote the number of balls in node $u$, initially setting $C(u) = 0$ for all $u$. Say that
$u$ is a {\em leaf} if $C(u) > 0$ and $C(v) = 0$ for all children $v$ of
$u$, and {\em internal} if 
$C(v) > 0$ for some proper descendant $v$, i.e., \(v < u\). We add $n$ balls labeled $\{1,\dots,n\}$ to $\cU$
one by one. The $j$-th ball is added by the following ``trickle-down'' procedure.
\begin{enumerate}
\item Add $j$ to the root.
\item While $j$ is at an internal node $u$, choose child $i$ with probability $V_{u,i}$, where
    $(V_{u,1}, \dots,V_{u,b})$ is the split vector at $u$, and move $j$ to child $i$.
\item If $j$ is at a leaf $u$ with $C(u) < s$, then $j$ stays at $u$ and we set $C(u) \leftarrow C(u) + 1$.

If $j$ is at a leaf with $C(u) = s$, then the balls at $u$ are distributed among $u$ and its
children as follows. We select $s_0\leq s$ of the balls uniformly at random to stay at $u$. Among
the remaining $s+1-s_0$ balls, we uniformly at random distribute $s_1$ balls to each of the $b$
children of $u$. Each of the remaining $s+1-s_0-bs_1$ balls is placed at a child node 
chosen independently at random according to the split vector assigned to \(u\).
This splitting process is repeated for any child which receives more than $s$ balls.
\end{enumerate}
For example, if we let \(b=2, s=s_0=1, s_{1}=0\) and \(\cV\) have the
distribution of \( (U, 1-U)\) where \(U \sim \Unif[0,1]\), then
we get the well-known binary search tree.

Once all $n$ balls have been placed in $\cU$, we obtain $T_n$ by deleting
all nodes $u$ such that the subtree rooted at $u$ contains no balls. Note
that 
 an internal node of $T_n$ contains exactly $s_0$ balls, while a leaf
contains a random amount in $\{1,\dots,s\}$. We assume, as previous authors,
that $\Prob{\exists i : V_i = 1} < 1$. 
 We can assume \(\cV\) has a symmetric (permutation invariant) distribution
without loss of generality, since a uniform random permutation 
of subtree order does not change the number of inversions. 

An equivalent definition of split trees is as follows.  Consider an infinite \(b\)-ary tree \(\cU\).
The split tree \(T_n\) is constructed by distributing \(n\) balls (pieces of information) among
nodes of \(\cU\). For a node \(u\), let \(n_u\) be the number of balls stored in the subtree rooted
at \(u\). Once \(n_u\) are all decided, we take \(T_n\) to be the largest subtree of \(\cU\) such
that \(n_u > 0\) for all \(u \in T_n\).
Let the split vector \(\cV \in [0,1]^b\) be as before. 
Let \(\cV_u = (V_{u,1},\dots,V_{u,b})\) be the independent copy of \(\cV\) assigned to \(u\). Let
\(u_1,\dots,u_b\) be the child nodes of \(u\). Conditioning on \(n_u\) and
\(\cV_u\),
 if $n_u\le s$, then $n_{u_i}=0$ for all $i$;
if $n_u>s$, then
\[
    (n_{u_1}, \dots, n_{u_{b}}) 
    \sim 
    \mathrm{Mult}(n-s_0-bs_{1}, V_{u,1}, \dots, V_{u, b}) 
    +
    (s_{1}, s_{1},\dots, s_{1}),
\]
where \(\mathrm{Mult}\) denotes multinomial distribution, and \(b, s, s_{0}, s_{1}\) are integers
satisfying \eqref{eq:split:para}. Note that \(\sum_{i=1}^b n_{u_i} \le n\) (hence the ``splitting'').
Naturally for the root \(\rho\), \(n_{\rho}=n\). Thus the distribution of \((n_{u}, \cV_{u})_{u \in V(\cU)}\) is completely defined.

\subsection{Outline}\label{sec:splitoutline}

In this section we outline how one can apply the \emph{contraction method} to prove \refT{thm:split} but leave
the detailed proof to \refS{sec:Neininger} and \refS{sec:split:mgf}. In Section \ref{sec:splitnodes} we state and
outline the proof of the corresponding theorem for inversions on nodes under
stronger assumptions.

Recall that in \eqref{eq:ballxyz}, we define
\begin{equation}
\bX_n = \frac{\bI(T_n) - \E{\bI(T_n)}}{n}, 
\quad \bY_n = \frac{\bI(T_n) - s_0\bU(T_n)/2}{n}, 
\quad \bW_n = \frac{\bU(T_n) - \E{\bU(T_n)}}{n}.
\end{equation}
Let \(\nBar = (n_{1},\dots,n_{b})\) denote the vector of the (random) number of balls in each of the
\(b\) subtrees of the root.  \citet{MR3025680} showed that, conditioning on \(\nBar\), 
\begin{equation}
\bW_n 
\eqdd 
\sum_{i=1}^b \frac{n_i}{n} \bW_{n_i} 
+ 
\frac{n-s_0}{n} 
+ 
\bTollN(\nBar), 
\quad 
\bTollN(\nBar) 
\eqd 
- \frac{\E{\bU(T_n)}}{n} +
\sum_{i=1}^b \frac{\E{\bU(T_{n_i})}}{n}
.
\label{eq:n:tollfunction:ball:W}
\end{equation}

We derive similar recursions for \(\bX_{n}\) and \(\bY_{n}\).
Conditioning on \(\nBar\), \(\bI(T_{n})\) satisfies the recursion
\begin{equation}\label{eq:ballrecursion}
\bI(T_n) \eqdd \bZ_\r + \sum_{i=1}^b \bI(T_{n_i}),
\end{equation}
where $\bZ_\r$ denotes the number of inversions involving balls contained in the root $\r$. 
Therefore, still conditioning on \(\nBar\), we have
\begin{align}
\bX_n 
&
\eqdd 
\sum_{i=1}^b \frac{n_i}{n} \bX_{n_i} 
+ 
\frac{\bZ_{\rho}}{n} 
-
\frac{\E{\bI(T_n)}}{n}
+ 
\sum_{i=1}^b \frac{\E{\bI(T_{n_{i}})}}{n}
\\
&
=
\sum_{i=1}^b \frac{n_i}{n} \bX_{n_i} 
+ 
\frac{\bZ_{\rho}}{n} 
- 
\frac{s_{0}}{2} \frac{\E{\bU(T_n)}}{n}
+ 
\frac{s_{0}}{2}\sum_{i=1}^b \frac{\E{\bU(T_{n_{i}})}}{n}
\\
&
=
\sum_{i=1}^b \frac{n_i}{n} \bX_{n_i} 
+ 
\frac{\bZ_{\rho}}{n} 
+ 
\frac{s_0}{2}\bTollN(\nBar)
,
\label{eq:n:tollfunction:ball:X}
\end{align}
where we use that
\begin{equation}
    \E{ \bI(T_n) \mid T_n}
    =
    \frac{s_{0}}{2}
    \bU(T_n)
    .
    \label{eq:inv:tp}
\end{equation}
(See the proof of \refL{lem:toll:func}.)
It follows also from \eqref{eq:inv:tp} that \(\bX_{n} = \bY_{n} + \frac{s_0}{2} \bW_{n}\) and
\begin{equation}
\bY_n 
\eqdd 
\sum_{i=1}^b \frac{n_i}{n} \bY_{n_i} 
+ 
\frac{\bZ_{\rho}}{n} 
-
\frac{s_0}{2} \frac{n-s_0}{n}
.
\label{eq:n:tollfunction:ball:Y}
\end{equation}

In Lemma \ref{lem:root:unif} below, we show that 
\begin{equation}\label{eq:balluniform}
    \frac{\bZ_\r}{n} \inLII U_1 + \dots + U_{s_0},
\end{equation}
where $U_1,\dots,U_{s_0}$ are independent and uniformly distributed in $[0, 1]$.
\citet{MR3025680} have shown that $\bTollN(\nBar) \inas \Toll(\cV)$, where
\begin{equation}
\m = -\sum_{i=1}^b \E{V_i\ln V_i}, \qquad \text{and} \qquad \Toll(\cV) = \frac{1}{\m} \sum_{i=1}^b V_i\ln V_i.
    \label{eq:tollfunction:ball}
\end{equation}
Together with $(n_1/n,\dots,n_b/n)\inas (V_1,\dots,V_b)$ (by the law of large number),
we arrive at the following fixed-point equations (already presented in \refT{thm:split})
\begin{equation}
    \left[
    \begin{aligned}
        & \bX \\
        & \bY \\
        & \bW
    \end{aligned}
    \right]
    \eql
    \left[
    \begin{aligned}
        & \sum_{i=1}^b V_i \bX^{(i)} + \sum_{j=1}^{s_0} U_j + \frac{s_0}{2} \Toll(\cV)  \\
        & \sum_{i=1}^b V_i \bY^{(i)} + \sum_{j=1}^{s_0} (U_j-1/2) \\
        & \sum_{i=1}^b V_i \bW^{(i)} + 1 + \Toll(\cV)
    \end{aligned}
    \right]
    .
    \label{eq:ballfix}
\end{equation}

For a random vector \(X \in \R^{d}\), let \(\norm{X}\) be the Euclidean norm of \(X\).  Let
\(\norm{X}_{2} \eqd \sqrt{\E{\norm{X}^2}}\).  Recall that \(\mZeroD\) denotes the set of
probability measures on \(\R^{d}\) with zero mean and finite second moment.  The Mallows metric on
\(\mZeroD\) is defined by
\[
\dM(\nu, \lambda) = \inf \left\{ \norm{X - Y}_{2} : X \sim \lambda, Y \sim \nu \right\}
.
\]
Using the {contraction method}, \citet{MR3025680} proved that $\bW_n \dMto \bW$, the unique solution
of the first equation of \eqref{eq:ballfix} in \(\mZeroI\).

We can apply the same contraction method to show that the vector $\bVecN \dMto \bVec$, the unique
solution of \eqref{eq:ballfix} in \(\mZeroIII\).  But we only outline the argument here since we
will actually use a result by \citet{MR1871564} which gives us a shortcut.
Assume that the independent vectors \(\bVecIx\), \(i=1,\dots,b\) share some common distribution \(\mu \in \mZeroIII\).  Let \(F(\mu) \in \mZeroIII\) be the
distribution of the random vector given by the right hand side of \eqref{eq:ballfix}.  Using a
coupling argument, we can show that for all \(\nu, \lambda \in \mZeroIII\), 
\[
    \dM(F(\nu), F(\lambda)) < c \dM(\nu, \lambda),
\]
where \(c \in (0,1)\) is a constant. Thus \(F\) is a
contraction and by Banach's fixed point theorem, \eqref{eq:ballfix} must have a unique solution
\(\bVec \in \mZeroIII\).
Finally, we can use a similar coupling argument to show that \(\bVecN \dMto \bVec\).

\subsection{Convergence in the Mallows metric}\label{sec:Neininger}

\begin{lemma}
    \label{lem:ball:Mallows}
    Let \(\bVecN\) and \(\bVec\) be as in \refT{thm:split}. Then
    \[
        \dM\left(\bVecN, \bVec \right)
        \to
        0.
    \]
\end{lemma}

We will apply Theorem 4.1 of \citet{MR1871564}, which summarizes sufficient conditions for the
contraction method outlined in the previous section to work.  Since the statement of the theorem is
rather lengthy, we do not repeat it here and refer the readers to the original paper.

\citeauthor{MR1871564}'s theorem implies that \(\bVecN \dMto \bVec\) if the following three
conditions are satisfied:
\begin{align}
    & \left(\frac{\bZ_{\rho}}{n}, \frac{n_1}{n},\dots, \frac{n_{b}}{n}, \bTollN(n) \right) \dMto
    \left(\sum_{j=1}^{s_0} U_j, V_1,\dots,V_{b}, \Toll(\cV) \right), 
    \qquad n \to \infty, 
    \label{cond:toll}
    \\
    & \sum_{i=1}^{b} \E{V_i^2} < 1, 
    \label{cond:V2}
    \\
    & \E{\ind_{[n_i \le \ell] \cup [n_i = n]} \left( \frac{n_{i}}{n} \right)^2} \to 0, \qquad n \to \infty,
    \label{cond:VZero}
\end{align}
for all \(\ell \ge 1\) and \(i = 1,\dots,b\).
(The three conditions correspond to (11), (12) and (13) in \cite{MR1871564}.)

Condition \eqref{cond:V2} is satisfied by the assumption that $\Prob{\exists i : V_i = 1} < 1$.
Since we assume that \(s_0 > 0\), the event \(n_i = n\) cannot happen. So the expectation in
\eqref{cond:VZero} is at most \( (\ell/n)^2 \to 0\) and this condition is also satisfied.
The last condition \eqref{cond:toll} follows from the following two lemmas.

\begin{lemma}
    \label{lem:toll:func}
    We have $\bTollN(\nBar) \inLII \Toll(\cV)$ and \(\sup_{n \ge 1} \bTollN(\nBar)\) is bounded deterministically.
\end{lemma}

\begin{proof}
We first derive an expression for the expected number of
inversions. Any internal node contains $s_0$ balls, so any ball at height $h$ has $s_0 \times h$ ancestral
balls. Let \(B(T_n)\) be the set of balls in \(T_n\).
Conditioning on \(T_n\), we have
$$
\E{\bI(T_n)\,|\,T_n} 
= \E{\left. \sum_{\b\in B(T_n)} |\{\b' : \b' < \b, \l(\b') > \l(\b)\}| \, \right| \, T_n} 
= \sum_{\b \in B(T_n)} \frac{s_0}{2} h(\b) 
= \frac{s_0}{2} \bU(T_n).
$$
Thus by \citet[Theorem 3.1]{MR3025680},
\begin{equation}\label{eq:firstmoment:ball}
\E{\bI(T_n)}
=
\frac{s_0}{2}
\E{\bU(T_n)} 
= 
\frac{s_0}{2}
\left[
\frac1\m n\ln n + n\varpi(\ln n) + o(n)
\right]
,
\end{equation}
with $\m$ as in \eqref{eq:tollfunction:ball}, where $\varpi$ is a continuous function of period $d = \sup\{a
\geq 0 : \Prob{\ln V_1 \in a\Z} = 1\}$. In particular, $\varpi$ is constant if $\ln V_1$ is {\em non-lattice}, meaning that \(d = 0\).

The convergence of the toll function can now be deduced from the same result on the total path length
from \cite{MR3025680}, but we include the short argument for completeness. Conditioning on the split
vector of the root $(V_1,\dots,V_n)$ and noting that $(n_1/n,\dots,n_b/n) \inas (V_1,\dots,V_b)$, we
have from \eqref{eq:n:tollfunction:ball:W}, \eqref{eq:firstmoment:ball},
\begin{align}
    \bTollN(\nBar) 
    & = - \frac{1}{\m} \ln n - \varpi(\ln n) + \sum_{i=1}^b \left( \frac{1}{\m} \frac{n_i\ln n_i}{n} + \frac{n_i}{n} \varpi(\ln n_i) \right) + o(1) \\
    & = \left(\sum_{i=1}^b \frac{1}{\m} \frac{n_i}{n}\ln\frac{n_i}{n} \right) + \left(\sum_{i=1}^b \frac{n_i}{n} \varpi\left( \ln V_i + \ln n \right)\right) -
        \varpi(\ln n)
        + o(1) \\
    & = \frac{1}{\m} \sum_{i=1}^b V_i\ln V_i + o(1),
\end{align}
where we use that \(\varpi\) is continuous and has the same period as \(\ln V_i\).
So we have
$$
\bTollN(\nBar) \inas \Toll(\cV) \eqd \frac{1}{\m} \sum_{i=1}^b V_i\ln V_i,
$$
without conditioning on $(V_1,\dots,V_b)$. Note that since for \(x_1,\dots,x_b \ge 0\) with
    \(\sum_{i=1}^b x_i = 1\), we have
\(
    \sum_{i=1}^b x_i \ln(x_i) \ge - \ln b
\) \cite[Theorem 3.1]{conrad2017},
both \(\bTollN(\nBar)\) and \(\Toll(\cV)\) are bounded deterministically.
Thus \(\bTollN(\nBar) \inLII \Toll(\cV)\) by the dominated convergence
theorem.
\end{proof}

\begin{lemma}
    \label{lem:root:unif}
    For \(i = 1,\dots,s_0\), let \(U_i\) be a \(\Unif[0,1]\) random variable independent of
    all other random variables.  Then there exists a coupling such that \({\bZ{}_{\rho}}/{n} \inLII
    \sum_{i=1}^{s_0} U_{i}\).
\end{lemma}

\begin{proof}
We have $\bZ_\r = \sum_{i=1}^{s_0} (\l_i - i)$, where $\l_1 < \l_2 < \dots < \l_{s_0}$ are the
labels for the balls in the root, chosen uniformly at random from $[n]$ without replacement. Indeed,
the ball with label $\l_i$ forms an inversion with the balls with labels $\{\l : \l < \l_i, \l \neq
\l_j \ \forall j<i\}$, a set of size $\l_i-i$.

Let \(\lambda_{i}' = \ceil{n U_{i}}\) for \(i=1,\dots,s_0\).
Then $\l_1', \dots,\l_{s_0}'$ are chosen independently and uniformly at random from $\{1,\dots,
n\}$. Define $\bZ_\r' = \sum_{i=1}^{s_0} (\l_i' - i)$. We couple $\bZ_\r'$ to $\bZ_\r$ by setting
$\bZ_\r = \bZ_\r'$ whenever all $\l_i'$ are distinct, and otherwise setting $\bZ_\r =
\sum_{i=1}^{s_0}(\l_i - i)$ for some distinct $\{\l_1,\dots,\l_{s_0}\}$ chosen uniformly at random. The
probability that $\l_i' = \l_j'$ for some $i\neq j$ is $O(1/n)$. (See the famous birthday problem
\cite[Example 3.2.5]{d10}.) Since $\bZ_\r \leq s_0n$ and
$\bZ_\r' \leq s_0n$, 
$$
\E{\left(\frac{\bZ_\r}{n} - \frac{\bZ_\r'}{n}\right)^2} 
\le \Prob{\exists i\neq j : \l_i' = \l_j'} \frac{4s_0^2n^2}{n^2} = O\bfrac1n.
$$
As  $|\l_i'/n - U_{i}| \le 1/n$, it is clear that
$\bZ_\r'/n = \sum_{i=1}^{s_0} (\l_i' - i)/n$ converges in the second moment to $\sum_{j=1}^{s_0} U_j$.
By the triangle inequality, this is also true for $\bZ_\r/n$.
\end{proof}

Since \( (n_1/n, \dots, n_b/n) \inas (V_1,\dots,V_b)\) and \(n_i/n \le 1\) for all \(i =
1,\dots,b\), the convergence is also in \(L^2\). This together with Lemma
\ref{lem:toll:func} 
and \ref{lem:root:unif} implies \eqref{cond:toll}. Therefore, it follows from Theorem 4.1 of
\citet{MR1871564} that \( \bVecN \dMto \bVec\).

\subsection{Convergence in moment generating function}
\label{sec:split:mgf}

To finish the proof of \refT{thm:split}, it remains to show following lemma.

\begin{lemma}\label{lem:mgflem}
    There exists a constant \(L \in (0, \infty]\) such that for all fixed $t \in \R^3$ with
    \(\norm{t} < L\),
\begin{equation}
\label{eq:mgflem}      
    \E{\exp\left(t \cdot \bVecN \right)} 
    \to 
    \E{\exp\left(t \cdot \bVec \right)} < \infty,
\end{equation}
    where \(\cdot\) denotes the inner product.
    If we further assume that \(\Prob{\exists i:V_i=1}=0\), then \(L = \infty\).
\end{lemma}

\begin{remark}
    The condition \(\Prob{\exists i:V_i=1}=0\) is necessary for \(L=\infty\). Assume the opposite. By
    \eqref{eq:ballfix}, for all \(t \in \R\),
    \begin{align*}
        \E{e^{t \bX}} 
        &
        \ge 
        \E{
            \left.
            t
            \left( 
                \sum_{i=1}^{b} U_{i}
                +
                \sum_{i=1}^{b} V_{i} \bX^{(i)}
                +
                \frac{s_0}{2}
                C\left( \cV \right)
            \right) \right|
            \exists i:V_i=1
        }
        \Prob{\exists i:V_i=1}
        \\
        &
        =
        \E{e^{t \sum_{i=1}^b U_{i}}}
        \Prob{\exists i:V_i=1}
        \E{e^{t \bX}}
        ,
    \end{align*}
    where \(U_{i}\) are independent \(\Unif[0,1]\).
    This implies that \(\E{e^{t \bX}} = \infty\) if we 
    chose \(t\) large enough such that \(\E{e^{t \sum_{i=1}^b U_{i}}} \Prob{\exists i:V_i=1} > 1\).
\end{remark}

The proofs of the next two lemmas are similar to Lemma 4.1 by \citet{MR1104413}, which deals with the
total path length of binary search trees.  However, we have extended the result to cover general
split trees.  Moreover, \refL{lem:mgf:J} can be applied not only to
inversions and the total path length, but also to any properties of split trees that satisfies the
assumptions.

\begin{lemma}
    \label{lem:U:bound}
    Let \(C_1 > 0\) be a constant.
    There exists a constant \(L\) such that for all \(t \in (-L,L)\), there exists \(K_{t} \ge 0\)
    such that
    \begin{equation}
        \E{\exp\{C_1 |t| + t^2 K_t U_n\}}
        \le 1,
        \qquad
        \text{for all}
        \quad
        n \in \N,
        \label{eq:fbound}
    \end{equation}
    where
    \[
        U_n \eqd -1 + \sum_{i=1}^b \bfrac{n_i}{n}^2.
    \]
    If we further assume that \(\Prob{\exists i:V_i=1}=0\), then \(L = \infty\).
\end{lemma}

\begin{proof}
    Let \(p = \Prob{\exists i:V_i = 1}\).
    Recalling the assumption that \(p < 1\), we can choose a constant \(\delta \in (0, 1-p)\).
    Then for \(\varepsilon\) small enough
    \[
        \Prob{-1+\sum_{i=1}^{b}V_{i}^2 \ge -\varepsilon} 
        \le
        \Prob{\exists i: V_{i}=1} + \frac{\delta}{2} = p + \frac{\delta}{2}
        .
    \]
    Since \(U_{n} \inas -1 + \sum_{i=1}^{b} V_{i}^{2}\), there exists \(n_0 \in \N\) such that
    \[
        \Prob{U_{n} \ge -\varepsilon}
        \le 
        \Prob{-1+\sum_{i=1}^{b}V_{i}^2 \ge -\varepsilon} 
        + 
        \frac{\delta}{2}
        \le p+\delta
        <1
        ,
        \qquad
        \text{for all}
        \quad
        n \ge n_{0}
        .
    \]
    Together with \(U_{n} \le 0\), the above inequality implies that for all \(n \ge n_0\),
    \(t \in (-L, L)\), and \(K_{t} \in \R\),
    \begin{equation}
        \E{\ind_{[U_{n} \ge -\varepsilon]} \exp\left( C_1 |t| + t^{2}K_{t} U_{n} \right)}
        \le
        e^{C_{1}L} \left( p+\delta \right)
        < 1
        ,
        \label{eq:fbound:a}
    \end{equation}
    if \(L\) is small enough.
    On the other hand,
we may assume that $t\neq0$ and then 
    \begin{equation}
        \E{\ind_{[U_{n} < -\varepsilon]} \exp\left( C_{1}|t| + t^{2}K_{t} U_{n}
            \right)}
        \le 
        \exp\left( C_{1}|t|-t^{2}K_{t} \varepsilon \right)
        <
        1- e^{C_{1}L}(p+\delta),
        \label{eq:fbound:b}
    \end{equation}
    if \(K_{t}\) is large enough.
    Together \eqref{eq:fbound:a} and \eqref{eq:fbound:b} implies \eqref{eq:fbound}.
    Note that if \(p = 0\), then \(L\) can be arbitrarily large.
\end{proof}

\newcommand{\Ain}{A_{n}^{(i)}}
\newcommand{\Jin}{J_{n}^{(i)}}
\newcommand{\Jini}{J_{n_i}^{(i)}}

\begin{lemma}\label{lem:mgf:J}
Let \((J_{n})_{n \ge 1}\) be a sequence of \(d\)-dimensional random vectors.
    Let \((\Jin)_{n \ge 1}\) for \(i = 1,\dots,b\) be independent copies of \( (J_{n})\).
    Let \(\Ain\) be a diagonal matrix with \(n_i/n\) on its diagonal.
    Let \( (B_{n})_{n \ge 1}\) be a sequence of random \(\N^{b} \to \R^{d}\) functions.
    Assume that conditioning on \(\nBar\),
    \begin{equation}
        J_{n} \eql \sum_{i=1}^{b} \Ain \Jini + B_{n}(\nBar).
        \label{eq:J:rec}
    \end{equation}
    Further assume that \(\sup_{n \ge 1} \norm{B_n(\nBar)} < C_1\) and \(\norm{J_{1}} < C_2\) deterministically for some
    constants \(C_1, C_2\) and that \(s_0 > 0\).
    Then there exists a constant \(L \in (0, \infty]\), such that
    for all \(t \in \R^d\) with \(\norm{t}<L\),
    there exists \(K_t \ge 0\), such that
    \begin{equation}
    \E{\exp(t\cdot J_{n})} \leq \exp( \norm{t}^2 K_t),
    \qquad
    \text{for all} 
    \quad
    n \in \N 
    .
    \label{eq:J:mgf}
    \end{equation}
    Moreover, if \(J_{n} \inlaw J^{*}\), then for all \(t \in \R^{d}\) with \(\norm{t} < L\),
    \begin{equation}\label{eq:J:mgf:converge}
        \E{\exp\left(t \cdot J_{n}\right)}
        \to
        \E{\exp\left(t \cdot J^{*} \right)}
<\infty.
    \end{equation}
If we further assume that \(\Prob{\exists i:V_i=1}=0\), then \(L = \infty\).
\end{lemma}

\begin{proof}
It follows from \refL{lem:U:bound} that there exists an \(L \in (0, \infty]\), such that
for all \(t\) with \(\norm{t} < L\), there exists \(K_{t} \ge 0\), such that
\begin{equation}
    \E{\exp\left(C_{1} \norm{t} + K_{t} \norm{t}^{2} U_{n}\right)}
    \le 1
    .
    \label{eq:mgf:factor}
\end{equation}
Now we use induction on \(n\).
Since \(\norm{J_1} < C_2\), we can increase \(K_t\) such that \eqref{eq:J:mgf} holds for \(n=1\).
Assuming that it holds also for all \(J_{n'}\) with \(n' < n\), we have
\begin{align}
    \E{\exp\left( t \cdot J_n \right)}
    &
    =
    \E{\exp\left( t \cdot B_{n}(\nBar) + t \cdot \sum_{i=1}^b \Ain \Jini \right)}
    \\
    &
    \le
    e^{C_{1} \norm{t}} \E{\sum_{i=1}^{b} K_{t} \left(\norm{t} \frac{n_{i}}{n}  \right)^{2} }
    \\
    &
    =
    e^{K_{t} \norm{t}^{2}}
    \E{\exp\left({C_{1} \norm{t} + K_{t} \norm{t}^{2} U_{n}}\right)}
    \le
    e^{K_{t} \norm{t}^{2}},
    \label{eq:bX:mgf1}
\end{align}
where we use \eqref{eq:mgf:factor} and that \(n_{i} < n\) for \(i=1\dots,b\) (since \(s_0 > 0\)).
The above inequality implies that
\( (e^{t \cdot J_{n}})_{n\ge1} \), are 
uniformly integrable (see \cite[Theorem\
5.4.2]{g13}).  
Therefore 
\(J_{n} \inlaw J^{*}\) implies \eqref{eq:J:mgf:converge}
(see \cite[Theorem\ 5.5.2]{d10}).
\end{proof}

\begin{proof}[Proof of Lemma \ref{lem:mgflem}]
    Let \(J_{n} = \bVecN\). Then \eqref{eq:n:tollfunction:ball:X}, \eqref{eq:n:tollfunction:ball:Y},
    \eqref{eq:n:tollfunction:ball:W} can be written as
    \begin{equation}
        J_{n} \eql \sum_{i=1}^{b} \Ain \Jini + B_{n}(\nBar),
    \end{equation}
    where \(\Ain\) for \(i =1,\dots,b\) are as in Lemma \ref{lem:mgf:J} and 
    \begin{equation}
        B_{n}(\nBar) = 
        \left[
            \frac{\bZ_{\rho}}{n} + \frac{s_0}{2} \bTollN(\nBar),
            \frac{\bZ_{\rho}}{n} - \frac{s_0}{2} \frac{n-s_0}{n},
            \frac{n-s_0}{n} + \bTollN(\nBar)
        \right]^{\bf T}
        ,
        \label{eq:B}
    \end{equation}
    where \({\bf T}\) denotes the transposition of a matrix.
    By \refL{lem:ball:Mallows}, 
    \( J_{n}\) converges in distribution to
    \(\bVec\).
    Note that \(\norm{B_{n}(\nBar)}\) is bounded.
    Therefore \refL{lem:mgf:J} implies that there exists an \(L \in (0,\infty]\) such
    that for all \(t \in \R^{3}\) with \( \norm{t} < L\), 
\(\E{e^{t \cdot J_{n}}} \to \E{e^{t \cdot \bVec}}<\infty\).
\end{proof}

\subsection{Split tree inversions on nodes}\label{sec:splitnodes}

We turn to node inversions in a split tree. The main challenge in this context is that the number
$N$ of nodes is random in general. Thus we will limit our analysis to split trees satisfying 
the following two assumptions
\begin{equation}
    \frac{N}{n} \inLII \alpha,
    \label{eq:tplassumption:1}
\end{equation}
and
\begin{equation}
    \E{\U(T_n)} = \frac{\a}{\m}n\ln n + n\varpi(\ln n) + o(n),
    \label{eq:tplassumption}
\end{equation}
for some  constant $\a \in (0, 1]$ and some continuous periodic function $\varpi$ with period $d = \sup\{a \geq 0 : \Prob{\ln V\in a\Z}=1\}$ (constant if
$d = 0$), with $\m = -\sum \E{V_1\ln V_1}$.

These two conditions are satisfied for many types of split trees. \citet{MR2878784} showed that if
\(\ln V_{1}\) is non-lattice, i.e., \(d=0\), then \(\E{N}/n = \alpha + o(1)\) and furthermore
\eqref{eq:tplassumption:1} holds. However, in the lattice case, \citet{MR995343} showed that, for
tries (split trees with \(s_0=0\) and \(s=1\)) with a fixed split vector \( (1/b,\dots,1/b)\),
\(\E{N}/n\) does not converge. Thus \eqref{eq:tplassumption:1} cannot be true for these trees.

Condition \eqref{eq:tplassumption} has been shown to be true for many types of split trees
including \(m\)-ary search trees \cite{MR905782, MR2035872, MR2504401, MR987097}. More specifically,
\citet{MR3025680} showed that in the non-lattice case, if \(\E{N}/n = \alpha + O(\ln^{-1-\varepsilon}
    n)\) for some \(\varepsilon > 0\), then \eqref{eq:tplassumption} is satisfied. However,
\citet{MR2735344} showed that even in the non-lattice case, there exist 
tries with some very special parameter values where \(\E{n}/n - \alpha\) tends to zero arbitrarily
slowly.

We have the following theorem that is similar to \refT{thm:split}.
\begin{theorem}
Assume the split tree $T_n$ satisfies \eqref{eq:tplassumption:1} and \eqref{eq:tplassumption} and define
$$
X_n = \frac{I(T_n) - \E{I(T_n)}}{n}, \quad Y_n = \frac{I(T_n) - \frac12 \U(T_n)}{n}, \quad W_n = \frac{\U(T_n) - \E{\U(T_n)}}{n}.
$$
Assume that \(\Prob{\exists i:V_i = 1} < 1\).
Let $\Toll(\cV)$ be as in \eqref{eq:tollfunction:ball}.
Let $\nVec$ be the unique solution in \(\mZeroIII\) for the
system of
fixed-point equations
\begin{equation}
    \left[
    \begin{aligned}
        & X \\
        & Y \\
        & W
    \end{aligned}
    \right]
    \eql
    \left[
    \begin{aligned}
        & \sum_{i=1}^b V_i X^{(i)} + \a U_0 + \frac{\a}{2}\Toll(\cV) \\
        & \sum_{i=1}^b V_i Y^{(i)} + \a \left(U_0-\frac{1}{2}\right) \\
        & \sum_{i=1}^b V_i W^{(i)} + \a + \a \Toll(\cV)
    \end{aligned}
    \right]
    .
\end{equation}
Here \( (V_{1},\dots,V_{b})\), \(U_{0}\),
\(
    (X^{(1)}, Y^{(1)}, W^{(1)}),
    \dots,
    (X^{(b)}, Y^{(b)}, W^{(b)})
\)
are independent,
with \(U_{0} \sim \Unif[0,1]\) and
\(\nVecI \sim \nVec\) for \(i = 1,\dots,b\).
Then \(\nVecN \dMto \nVec\).
If \(s_0 > 0\), then \(\nVecN\) also converges to \(\nVec\) in moment generating function
within a neighborhood of the origin.
\end{theorem}

The convergence in Mallows metric again follows from \citet[Theorem 4.1]{MR1871564}. We leave the
details to the reader as it is rather similar to inversions on balls.
However, we emphasize that the assumption \eqref{eq:tplassumption} is needed to argue that
\begin{equation}
    \TollN(\nBar) \eqd -\frac{\E{\U(T_n)}}{n} + \frac1n\sum_{i=1}^b \E{\U(T_{n_i})} 
    \inLII
    \frac{\a}{\m} \sum_{i=1}^b V_i\ln V_i = \a \Toll(\cV).
    \label{eq:c:n:w}
\end{equation}
For convergence in moment generating function, note that \(s_0 > 0\) implies \(N \le n\) and
\(Z_{\rho}/n \le 1\). Therefore, we can again apply \refL{lem:mgf:J} as in \refS{sec:split:mgf}.

\section{A sequence of conditional Galton--Watson trees}\label{sec:galtonwatson}

Let $\xi$ be a random variable with $\E{\xi} = 1$, 
$\va\ \xi = \s^2 <\infty$,
and $\E{e^{\a\xi}}<\infty$ for some $\a>0$, 
(The last condition is used in the proof below, but is presumably not
necessary.) 
Let $G^\xi$ be a (possibly infinite) Galton--Watson tree with offspring distribution $\xi$. The {\em conditional Galton--Watson tree} $T^\xi_n$ on $n$ nodes is given by
$$
\Prob{T^\xi_n = T} = \Prob{G^\xi = T\ \middle|\ G^\xi \text{ has $n$ nodes}}
$$
for any rooted tree $T$ on $n$ nodes. The assumption $\E{\xi} = 1$ is justified by noting that if $\zeta$
is such that $\Prob{\xi = i} = c\th^i\Prob{\zeta=i}$ for all $i\geq 0$ then $T^\xi_n$ and $T^\z_n$
are identically distributed; 
hence it is typically possible to replace an
offspring distribution $\zeta$ by an equivalent one with mean 1, 
see \cite[Sec.\ 4]{j12}. 

We fix some $\xi$ and drop it from the notation, writing $T_n =
T^\xi_n$.

In a fixed tree $T$ with root $\r$ and $n$ total nodes, for each node $v\neq \r$ let $\QQ_v\sim\Unif(-1/2, 1/2)$, all independent, and let $\QQ_\r = 0$. For each node $v$ define
$$
\F_v \eqd \sum_{u \leq v} \QQ_u, \quad 
\text{and let}\quad J(T) \eqd \sum_{v \in T} \F_v.
$$
In other words, \(\F_u\) is the sum of $\QQ_v$ for all $v$ on the path from the root to \(u\).
For each $v\neq \r$ also define $Z_v = \floor{(\QQ_v + 1/2)z_v}$, where $z_v$ denotes the size of the
subtree rooted at $v$. Then $Z_v$ is
uniform in $\{0,1,\dots,z_v-1\}$, and by Lemma \ref{lem:independence}, the quantity
$$
I^*(T) \eqd \sum_{v\neq \r} \bigpar{Z_v - \E{Z_v}}
$$
is equal in distribution to the centralized number of inversions in the tree
$T$, ignoring inversions 
involving $\r$. The main part \eqref{gw} of Theorem \ref{thm:galtonwatson} will follow from arguing that for a conditional
Galton--Watson tree $T_n$,
\begin{equation}\label{eq:snakelimit}
\frac{J(T_n)}{n^{5/4}} \dto \YY \eqd \frac{1}{\sqrt{12\s}}\sqrt{\eta} \cN.
\end{equation}
Indeed, under the coupling of $\QQ_v$ and $Z_v$ above,
$$
J(T_n) = \sum_v \F_v
= \sum_v \sum_{u:u \le v} \QQ_u
= \sum_u \QQ_u \sum_{v:u \le v} 1
= \sum_u \QQ_u z_u
\leq \sum_{u \neq \r} \left(Z_u - \frac{z_u}{2} + 1\right)
< n + I^*(T_n)
,
$$
and similarly $J(T_n) > I^*(T_n) - n$. As $\r$ contributes at most $n$ inversions to $I(T_n)$, it
follows from the triangle inequality that $|J(T_n) - (I(T_n)-\U(T_n)/2)| \leq 2n = o(n^{5/4})$.
Thus \eqref{eq:snakelimit}, once proved, will imply that
\[
    \YY_n 
    \eqd 
    \frac{I(T_n)-\U(T_n)/2}{n^{5/4}}
    =
    o(1) + \frac{J(T_n)}{n^{5/4}}
    \inlaw
    \YY.
\]

The quantity $J(T_n)$ and the limiting distribution \eqref{eq:snakelimit}
have been considered by several authors.
In the interest of keeping this section
self-contained, we will now outline the proof of \eqref{eq:snakelimit} which
relies on the concept 
of a {\em discrete snake}, a random curve which under proper rescaling
converges to a \emph{Brownian snake}, a curve related to a standard
Brownian excursion. 
This convergence was shown by \citet{MR2033198}, and later in more generality by
\citet{jm05}, whose notation we use.

Define $f:\{0,\dots,2(n-1)\}\to V$ by saying that $f(i)$ is the location of a depth-first search
(under some fixed ordering of nodes) at stage $i$, with $f(0) = f(2(n-1)) = \r$. Also define $V_n(i) =
d(\r, f(i))$ where $d$ denotes distance. The process $V_n(i)$ is called the depth-first walk, the
Harris walk or the tour of $T_n$. For non-integer values $t$, $V_n(t)$ is given by linearly
interpolating adjacent values. See Figure \ref{fig:dfw}.
\begin{figure}[ht]
\centering
\begin{subfigure}{0.4\textwidth}
\includegraphics[page=1,width=\textwidth]{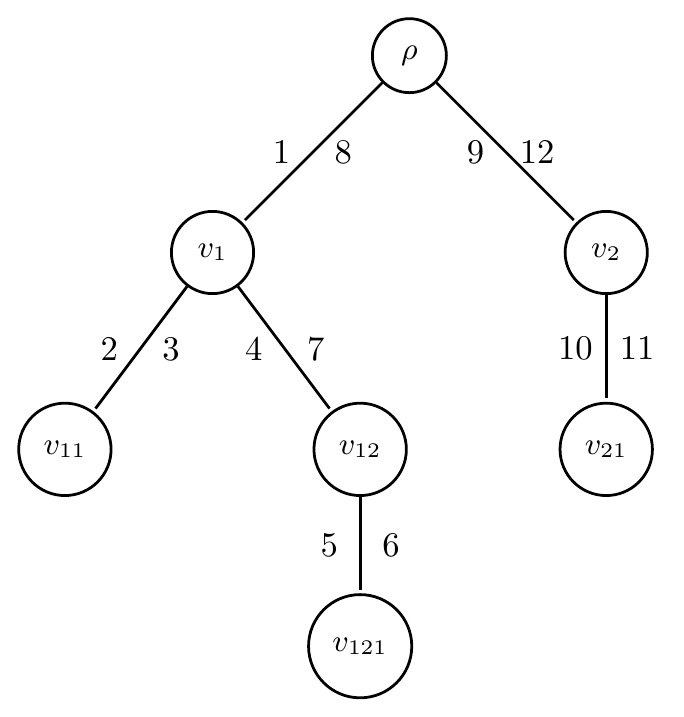}
\end{subfigure}
\begin{subfigure}{0.5\textwidth}
\includegraphics[page=2,width=\textwidth]{gw.pdf}
\end{subfigure}
\caption{The depth-first walk $V_n(t)$ of a fixed tree.}
\label{fig:dfw}
\end{figure}

Finally, define $R_n(i) \eqd \F_{f(i)}$ to be the value at the vertex visited after $i$ steps. For non-integer values $t$, $R_n(t)$ is defined by linearly interpolating the integer values. Also define $\widetilde{R}_n(t)$ by $\widetilde{R}_n(t) \eqd R_n(t)$ when $t\in \{0,1,\dots,2n\}$, and
$$
\widetilde{R}_n(t) \eqd \left\{\begin{array}{lcl}
R_n(\lfloor t\rfloor), & \text{if} & V_n(\lfloor t\rfloor) > V_n(\lceil t\rceil), \\
R_n(\lceil t\rceil), & \text{if} & V_n(\lfloor t\rfloor) < V_n(\lceil t\rceil).
\end{array}\right.
$$
In other words, $\widetilde{R}_n(t)$ takes the value of node $f(\lfloor t\rfloor)$ or $f(\lceil
t\rceil)$, whichever is further from the root. We can recover $J(T_n)$ from $\widetilde{R}_n(t)$ via
$$
2J(T_n) = \int_0^{2(n-1)} \widetilde{R}_n(t)dt.
$$
Indeed, for each non-root node $v$ there are precisely two unit intervals during which $\widetilde{R}_n(t)$ draws its value from $v$, namely the two unit intervals during which the parent edge of $v$ is being traversed. Now, since $\QQ_v\sim \text{Unif}(-1/2,1/2)$ we have $|R_n(i) - R_n(i-1)| \leq 1/2$ for all $i>0$ and
$$
\frac{J(T_n)}{n^{5/4}} = \frac{1}{2n^{5/4}} \int_{0}^{2(n-1)} \widetilde{R}_n(t) \mathrm dt =
\frac{1}{2n^{5/4}} \int_0^{2(n-1)} R_n(t) dt + O(n^{-1/4}) = \int_0^1r_n(s)ds + o(1),
$$
where $r_n(s) \eqd n^{-1/4}R_n(2(n-1)s)$. 
Also normalize $v_n(s) \eqd n^{-1/2}V_n(2(n-1)s)$. Theorem 2 of
\cite{jm05} 
(see also \cite{MR2033198})
states that $(r_n, v_n) \dto (r, v)$ in $C[0,1] \times C[0,1]$, with $r,v$ to be defined
shortly. 

Before defining $r$ and $v$, we will briefly motivate what they ought to be. Firstly, as the
offspring distribution $\xi$ of $T_n$ satisfies $\E{\xi} = 1$, we expect the
tour $V_n$ to be roughly a
random walk with zero-mean increments, conditioned to be non-negative and return to the origin at
time $2(n-1)$, and the limiting law $v$ ought to be a Brownian excursion (up
to a constant scale factor).
Secondly, consider a node $u$
and the path $\r = u_0,u_1\dots,u_d = u$, where $d$ is the depth of $u$. We can define a random walk
$\F_u(t)$ for $t=0,\dots,d$ by $\F_u(0) = 0$ and $\F_u(t) = \sum_{i=1}^t \QQ_{u_{i}}$ for $t > 0$,
noting that $\F_u = \F_u(d)$. Under rescaling, the random walk $\F_u(t)$ will behave like Brownian
motion. For any two nodes $u_1,u_2$ with last common ancestor at depth $m$, the processes $\F_{u_1},
\F_{u_2}$ agree for $t=0,\dots,m$, while any subsequent increments are independent. Hence
$\text{Cov}(\F_{u_1}, \F_{u_2}) = cm$ for some constant $c > 0$. Now, for any
$i,j\in\{0,\dots,2(n-1)\}$, the nodes $f(i), f(j)$ at depths $V_n(i), V_n(j)$ have last common ancestor
$f(k)$, where $k$ is such that $V_n(k)$ is minimal in the range $i\leq k\leq j$. Hence $r(s)$ should
be normally distributed with variance given by $v(s)$, and the covariance of $r(s), r(t)$
proportional to $\min_{s\leq u\leq t} v(u)$.

We now define $r, v$ precisely. If $\va\ \xi = \s^2$, then 
$v(s) \eqd 2\s^{-1} e(s)$, where $e(s)$ is a
standard Brownian excursion, as shown by Aldous \cite{a91, MR1207226}. Conditioning on $v$, we define $r$ as a
centered Gaussian process on $[0,1]$ with
$$
\text{Cov}(r(s), r(t) \mid v) = \frac{1}{12}\min_{s\leq u\leq t} v(u) 
=
\frac{1}{12\s} C(s, t),
\qquad s\le t.
$$
The constant $1/12$ appears as the variance of the random increments $\QQ_v$. Again, Theorem 2 of \cite{jm05} states that $(r_n, v_n) \dto (r, v)$ in $C[0,1]^2$. We conclude that
$$
\lim_{n\to\infty} \frac{J(T_n)}{n^{5/4}} =\int_0^1r_n(t)dt+o(1)
\dto \int_0^1 r(t)dt \eqd Y.
$$

This integral is the object of study in \cite{MR2108865}, wherein it is shown that
$$
Y \eqd \int_0^1 r(t)dt \eql \frac{1}{\sqrt{12\s}} \sqrt{\eta}\ \cN
,
$$
where $\cN$ is a standard normal variable, $\eta$ is given by
$$
\eta = \int_{[0, 1]^2} C(s, t) \mathrm ds\ \mathrm dt,
$$ 
and $\eta, \cN$ are independent. The odd moments of $Y$ are zero, as this is the case for $\cN$, and by \cite[Theorem 1.1]{MR2108865}, for $k\geq 0$
$$
\E{Y^{2k}} = \frac{1}{(12\s)^k} \frac{(2k!)\sqrt{\pi}}{2^{(9k-4)/2}\G((5k-1)/2)} a_k,
$$
where $a_1 = 1$ and for $k\geq 2$,
$$
a_k = 2(5k-4)(5k-6)a_{k-1} + \sum_{i=1}^{k-1}a_ia_{k-i}.
$$
In particular (\cite[Theorem 1.2]{MR2108865}),
$$
\E{Y^{2k}} \sim \frac{1}{(12\s)^k}\frac{2\pi^{3/2}\b}{5} (2k)^{1/2} (10e^3)^{-2k/4} (2k)^{\frac34 \cdot 2k},
$$
as $k\to\infty$, where $\b = 0.981038\dots$. Further analysis of the moments
of \(\eta\) and $Y$, including the moment
generating function and tail estimates, can be found in \cite{MR2108865}.

\begin{remark}
Conditioning on the value of $\eta$, the random variable $Y$ has variance
$\eta/(12\sigma)$. The
random variable $\eta$ can be seen as a scaled limit of the second common path length $\U_{2}(T_n)$,
which appeared in our earlier discussion on cumulants. Indeed, 
recall that 
\(\U_{2}(T_n) \eqd \sum_{u, v\in T_n} c(u, v)\),
where $c(u,v)$ denotes
the number of common ancestors of $u, v$. 
\end{remark}

\subsection{Convergence of the moment generating function}

The last bit of Theorem \ref{thm:galtonwatson} which remains to be proved is that \(\E{e^{t\YY_n}} \to
\E{e^{t\YY}}\) for all fixed \(t \in \R\). Since we have already shown
\(\YY_n \inlaw \YY\), we can apply the Vitali convergence theorem once we
have shown that the sequence \(e^{t \YY_n}\) is uniformly integrable. This
follows from the following lemma. 

\begin{lemma}
    \label{lem:mgflem:gw}
    For all \(n \in \N\) and \(t \in \R\), there exist positive constants \(C_1\) and \(c_1\) which do not
    depend on \(n\) such that
\[
        \E{e^{t\YY_n}} \le C_1 e^{c_1{t^4}}.
    \]
\end{lemma}
\begin{proof}
    Conditioned on \(T_n\), we have by \eqref{mgfIT*}
\begin{align*}
        \E{e^{tY_n}\mid T_n}
        \le
\exp\lrpar{\frac18\Bigpar{\frac{t}{n^{5/4}}}^2\U_2(T_n)}
=
\exp\lrpar{\frac{t^2}8\cdot\frac{\U_2(T_n)}{n^{5/2}}}.
    \end{align*}
    By \eqref{Uk}, we have
    \[
        \U_2(T_n) = \sum_{u, v \in T_n} c(u, v) \le n^2 (H_{n}+1),
    \]
    where \(H_n\) denotes the height of \(T_n\). It follows that
    \begin{align*}
        \E{e^{t\YY_n}}
        \le
        \E{\exp\left( \frac{\U_2(T_n)}{n^{5/2}} t^{2} \right)}
        \le
        \E{\exp\left( \frac{H_n +1}{\sqrt{n}} t^{2} \right)}
        \le
        e^{t^2}
        \E{\exp\left( \frac{ H_n}{\sqrt{n}} t^{2} \right)}
        .
    \end{align*}
    The random variable \(H_n\) has been well-studied. In particular, \citet{adj13} showed that
    there exist positive constants \(C_2\) and \(c_2\) such that
    \[
        \Prob{H_n > x} \le C_{2} \exp\left( -c_2 \frac{x^2}{n} \right),
    \]
    for all \(n \in \N\) and \(x \ge 0\).
    Therefore, we have
    \begin{align*}
        \E{\exp\left( \frac{ H_n}{\sqrt{n}} t^2 \right)}
        =
        1 + 
        \int_{0}^{\infty} e^{x} \Prob{\frac{ H_n}{\sqrt{n}} t^2 > x} \mathrm{d} x
        \le
        1 + 
        \int_{0}^{\infty} e^{x} C_{2} \exp\left( -c_2 \frac{x^2}{t^4} \right) \mathrm{d} x
        \le 1+ C_{1}  t^2 e^{c_3{t^4}}
    \end{align*}
    for some positive constants \(c_3\) and \(C_1\).
    (For the equality in the above computation, see \cite[pp.\ 56]{d10}.)
    Thus the lemma follows.
\end{proof}



\end{document}